\documentclass[2p]{article}
\usepackage{amssymb,amsmath,amsthm}
\usepackage{indentfirst}
\usepackage{exscale}
\usepackage{relsize}
\usepackage[numbers,sort&compress]{natbib}

\usepackage{geometry}
\usepackage{color}
\usepackage{hyperref}
\usepackage{aliascnt}  % <--- 新增
\newcommand{\R}{\mathbb{R}}
\newtheorem{theorem}{Theorem}[section]

\theoremstyle{definition}
\newaliascnt{definition}{theorem}
\newtheorem{definition}[definition]{Definition}
\aliascntresetthe{definition}

\newaliascnt{example}{theorem}

\aliascntresetthe{example}

\newaliascnt{algorithm}{theorem}

\aliascntresetthe{algorithm}

\newaliascnt{conclusion}{theorem}

\aliascntresetthe{conclusion}

\newaliascnt{problem}{theorem}

\aliascntresetthe{problem}

\newaliascnt{lemma}{theorem}
\newtheorem{lemma}[lemma]{Lemma}
\aliascntresetthe{lemma}

\newaliascnt{proposition}{theorem}
\newtheorem{proposition}[proposition]{Proposition}
\aliascntresetthe{proposition}

\newaliascnt{corollary}{theorem}
\newtheorem{corollary}[corollary]{Corollary}
\aliascntresetthe{corollary}

\newaliascnt{question}{theorem}

\aliascntresetthe{question}

\newaliascnt{remark}{theorem}
\newtheorem{remark}[remark]{Remark}
\aliascntresetthe{remark}

\numberwithin{equation}{section}
\usepackage{cleveref}% 交叉引用增强包

\crefname{theorem}{theorem}{theorems}
\Crefname{theorem}{Theorem}{Theorems}
\crefname{lemma}{lemma}{lemmas}
\Crefname{lemma}{Lemma}{Lemmas}
\crefname{corollary}{corollary}{corollaries}
\Crefname{corollary}{Corollary}{Corollaries}
\crefname{proposition}{proposition}{propositions}
\Crefname{proposition}{Proposition}{Propositions}
\crefname{definition}{definition}{definitions}
\Crefname{definition}{Definition}{Definitions}
\crefname{remark}{remark}{remarks}
\Crefname{remark}{Remark}{Remarks}

\begin{document}
\title{\Large\bf{ Existence of a nonnegative bound state for a higher-order logarithmic Schr\"{o}dinger equation}}
\date{}
\author {Xin Ou$^{1}$, \ Xingyong Zhang$^{1,2}$\footnote{Corresponding author, E-mail address: zhangxingyong1@163.com}, \ Yanhong Li$^{1,3}$ \\
{\footnotesize $^1$Faculty of Science, Kunming University of Science and Technology, Kunming, Yunnan, 650500, P.R. China.}\\
{\footnotesize $^{2}$Research Center for Mathematics and Interdisciplinary Sciences, Kunming University of Science and Technology,}\\
{\footnotesize Kunming, Yunnan, 650500, P.R. China.}\\
{\footnotesize $^3$	City College, Kunming University of Science and Technology, Kunming, Yunnan, 650051, P.R. China.}\\
}

\date{}
\maketitle

\begin{center}
\begin{minipage}{15cm}
\par
\small {\bf Abstract:} In this paper, we study the existence of nonnegative bound states for the higher-order logarithmic Schrödinger equation 
$$
-\alpha\Delta u+V(x)u-\beta u\ln|u| = \gamma u(\ln|u|)^m, \qquad u\in H^1(\mathbb R^N),
$$
where $\alpha,\beta,\gamma>0$, $m$ is an odd positive integer and $V$ is a positive bounded potential converging to a constant at infinity. The simultaneous presence of the first- and higher-order logarithmic terms leads to a nonsmooth variational structure and additional compactness difficulties. We introduce a family of power-law approximations and derive estimates that are uniform as the approximation exponent tends to the logarithmic limit. Under a structural condition on the autonomous nonlinearity, we obtain uniqueness, up to translations, of the nonnegative autonomous profile with connected positivity set and a corresponding energy gap below the two-profile threshold. The small-amplitude behavior exhibits a maximum-principle/compact-support dichotomy: the case $m=1$ yields positivity, whereas the higher odd orders fall into the compact-support regime for the autonomous profile. We then establish a profile decomposition for constrained Palais-Smale sequences and construct a barycenter-based min-max level below the splitting threshold. This prevents loss of mass through multiple profiles and allows us to pass to the logarithmic limit. Under the stated structural and energy conditions, we obtain a nontrivial nonnegative bound state of the original equation.

\par
{\bf Keywords:} Logarithmic Schr\"odinger equation, higher-order logarithmic nonlinearity, power-type approximation, bound state solution.
\par
{\bf 2020 Mathematics Subject Classification.} 35B09; 35J10; 35J20; 35Q55.
\end{minipage}
\end{center}
\allowdisplaybreaks
\vskip2mm
\section{Introduction }
\setcounter{equation}{0}
\noindent
Due to its important applications in quantum mechanics and nonlinear wave theory, the nonlinear Schr\"{o}dinger equation has long been the subject of extensive research. Among various types of solutions, standing wave solutions play an important role in understanding the qualitative behavior and stability properties of solutions. A standing wave solution is usually considered in the form
$$
\psi(t,x)=\text{e}^{-i\omega t}u(x),
$$
which reduces the corresponding evolution equation into a nonlinear elliptic problem.

The logarithmic Schr\"odinger equation is a special class of nonlinear Schr\"odinger equations. The classical logarithmic Schr\"odinger equation
\[
-\Delta u+V(x)u=u\ln u^2
\]
was introduced by Bialynicki-Birula and Mycielski \cite{Bialynicki-Birula1976} and has attracted considerable attention due to its applications in quantum mechanics, quantum optics, nuclear physics, transport and diffusion phenomena, open quantum systems, effective quantum gravity, superfluid theory, and Bose–Einstein condensation \cite{Zloshchastiev2010}. Compared with the standard power-law nonlinearities, the logarithmic term exhibits a singular behavior near the origin, which leads to a different variational structure from the usual nonlinear Schr\"odinger equations.

The singularity of the logarithmic nonlinearity brings essential difficulties to the variational analysis. In particular, the corresponding energy functional is not a standard $C^1$ functional on $H^1(\mathbb{R}^N)$, and therefore the classical critical point theory cannot be directly applied. To overcome this difficulty, under the variational framework introduced by Cazenave \cite{Cazenave1980}, several approaches have been developed, including the penalization method \cite{Alves2020}, the decomposition method \cite{Avenia2014}, the power-type approximation method \cite{wangzq2019}, and the Lyapunov--Schmidt reduction method \cite{wangzq2023}. Based on these approaches, various problems concerning the existence and multiplicity of solutions, as well as multi-peak and nodal solutions for the classical logarithmic Schr\"odinger equations, have been investigated; see \cite{Alves2022,Alves2024,Tanaka2017,wangj2025,zhangcx2020,zhangcx20202,Squassina2015,Squassina2017,jic2016,fengwx2024,shuai2019,shuai2021,shaom2024,hez2024} and the references therein. For the power-type approximation method proposed by Wang-Zhang \cite{wangzq2019}, the main idea is to transform the original logarithmic problem into a family of smooth power-type problems. More precisely, they study the convergence behavior as
$$
\frac{s^{p-2}-1}{p-2}\rightarrow \ln s,
\qquad p\rightarrow2^+,
$$
and then apply standard variational methods to the resulting power-type nonlinear problems. More recently, Gallo-Mosconi-Squassina \cite{Gallo2026} studied positive Dirichlet solutions of the classical logarithmic Schr\"odinger equation on bounded convex domains through a Lane-Emden power-law approximation. By letting the power exponent tend to the logarithmic regime, they obtained convergence to a solution of the logarithmic equation and established logarithmic concavity of a limiting solution. This recent result further illustrates the usefulness of power-law approximations for qualitative questions in logarithmic Schr\"odinger problems.

Bound state solutions of the classical logarithmic Schr\"odinger equations have also attracted considerable attention. For instance, Feng-Tang-Zhang \cite{fengwx2024} established the existence of positive bound state solutions for a class of logarithmic Schr\"odinger equations by combining variational methods with the barycenter technique. Both Ikoma-Tanaka-Wang-Zhang \cite{Ikoma2021}  and Zhang-Zhang\cite{zhangcx2020} employed the penalization method to study the existence and concentration behavior of bound states for the classical logarithmic Schr\"odinger equations. The former mainly considered potentials with suitable well structures and established the concentration of semiclassical states near the minimum region of the potential, while the latter further investigated the case where the potential is allowed to be unbounded below at infinity with a decay rate controlled by at most quadratic growth. Moreover, the existence and concentration behavior of positive bound states for general nonlinear Schr\"odinger equations have also been extensively studied, see, for example, \cite{Cerami2005,wangxf1993,wangxf1997,Pino1998}. These results provide important references for the analysis of related problems in logarithmic Schr\"odinger equations.

Recently, higher-order logarithmic nonlinearities have attracted increasing attention due to their applications in correlated systems and related mathematical models. Li-Ye-Cui \cite{lir2024} introduced the following higher-order logarithmic Schr\"odinger equation:
\begin{eqnarray*}
	i\frac{\partial \psi(x, t)}{\partial t}
	=
	-\alpha \Delta \psi(x, t)
	+W(x)\psi(x, t)
	+\beta \psi(x, t)\ln|\psi(x, t)|
	+\gamma \psi(x, t)(\ln|\psi(x, t)|)^m,
\end{eqnarray*}
which was derived by the authors through introducing a nonadditive entropy capable of describing N-body correlations, in analogy to the Shannon information entropy (an additive entropy) proposed by von Neumann \cite{Brasher1991}. It is primarily employed to investigate system correlations and spatiotemporal variations of the order parameter, and can provide a fundamental theoretical framework and physical insights for nonequilibrium phase transitions and nonlinear wave propagation in complex physical systems. In \cite{lir2024}, the authors utilized the generalized Madelung transformation and the Laplace transform to effectively transform logarithmic nonlinearity into manageable power-law nonlinearity. They investigated the analytical solutions and numerical simulations of the system’s dynamical behavior in zero- to two-dimensional spaces, revealing a novel phase transition involving symmetry breaking and symmetry preservation. In addition, they simulated the behavior of femtosecond pulse sequences and cellular electrophysiological spikes.

A typical autonomous higher-order logarithmic Schr\"odinger equation is 
\[
-\Delta u=u(\ln|u|)^m,
\]
which can be regarded as a natural extension of the classical logarithmic Schr\"odinger equation. An-Fang \cite{anx2025} established the existence of a positive ground state for the autonomous higher-order logarithmic
Schr\"odinger equation for admissible logarithmic orders. They also proved the radial symmetry of positive solutions and obtained the uniqueness and non-degeneracy of positive ground states when the logarithmic order is
sufficiently close to one. Their results reveal the structure of the autonomous limiting problem and play an important role in the analysis of nonautonomous problems.

However, in contrast to the autonomous problems studied in \cite{anx2025} and the classical logarithmic Schr\"odinger equation discussed earlier, this paper primarily investigates the higher-order logarithmic Schr\"odinger equation containing a potential function $V(x)$, a first-order logarithmic term $u\ln|u|$, and higher-order logarithmic terms $u(\ln|u|)^m$, the form of the equation is as follows:
\begin{eqnarray}
	\label{l1: eq1-1}
	-\alpha\Delta u+ V(x)u- \beta u\ln {|u|}
	=
	\gamma u(\ln {|u|})^m,
	\qquad
	u\in H^1(\mathbb R^N),
\end{eqnarray}
where $\alpha, \beta, \gamma$ are positive constants, $V$ is a continuous and bounded potential function in $\R^N$, $m$ is an odd positive integer and $N\geq 1$. 

The simultaneous presence of the first-order logarithmic term, the higher-order logarithmic term, and the asymptotically constant potential gives rise to several substantial difficulties in the analysis of equation \eqref{l1: eq1-1}. These difficulties cannot be handled by a direct combination of the arguments developed for the first-order logarithmic problem in \cite{fengwx2024} and for the autonomous higher-order logarithmic problem in \cite{anx2025}. The first difficulty concerns the variational structure. As in the classical logarithmic Schr\"odinger equation, the logarithmic nonlinearity makes the associated energy functional fail to be a standard $C^1$ functional on $H^1(\mathbb R^N)$, which precludes the direct application of classical critical point theory. Although this difficulty can be overcome by introducing a power-law approximation, the present approximation contains both the first-order term $\frac{|u|^{p-2}-1}{p-2}$ and its higher-order power $\left(\frac{|u|^{p-2}-1}{p-2}\right)^m$. Consequently, the estimates required here involve a coupling of the first- and higher-order logarithmic effects and are substantially more delicate than those arising in the first-order model of \cite{fengwx2024}. On the other hand, unlike the autonomous higher-order problem considered in \cite{anx2025}, the present equation also contains the potential term and the first-order logarithmic term. Therefore, the uniform estimates obtained for the pure higher-order logarithmic nonlinearity cannot be directly transferred to the present setting. In particular, additional estimates uniform with respect to $p\to2^+$ are needed in order to pass to the logarithmic limit. 

The second difficulty concerns the autonomous limiting problem. In the first-order setting of \cite{fengwx2024}, the autonomous logarithmic equation has an explicitly known positive ground state, whose uniqueness up to translations is available. In contrast, the autonomous problem arising here contains both the first-order and the higher-order logarithmic terms, and its relevant nonnegative single-bump ground state is no longer explicitly available. Moreover, for odd $m\geq3$, the small-amplitude regime falls on the compact-support side of the strong-maximum-principle/compact-support dichotomy. It is therefore necessary to analyze the structure of the limiting nonlinearity and establish the uniqueness, up to translations, of the nonnegative autonomous profile with connected positivity set, which is essential for identifying the possible profiles appearing in the subsequent compactness analysis. 

The third difficulty is the loss of compactness caused by the asymptotically constant potential. As in \cite{fengwx2024}, Palais-Smale sequences may escape to infinity and decompose into translated nontrivial profiles of the autonomous limiting problem. However, in the present higher-order setting, obtaining a profile decomposition alone is not sufficient. One must further derive a quantitative energy estimate that rules out the coexistence of two or more nontrivial profiles. In particular, it is crucial to show that the relevant min--max level lies strictly below twice the least energy carried by a nontrivial profile. Establishing this strict energy threshold in the presence of the higher-order logarithmic term requires additional estimates that have no counterpart in the first-order model.

To overcome the above difficulties, we first introduce the power-law approximation, and consider the corresponding family of smooth problems for $p>2$ close to $2$. In contrast to the first-order approximation used in \cite{fengwx2024} and the pure higher-order autonomous approximation studied in \cite{anx2025}, the present approximating problem simultaneously contains the first-order term $L_p(u)$, the higher-order term $L_p(u)^m$, and the nonconstant potential $V(x)u$. We therefore establish a collection of estimates that are uniform with respect to $p\to2^+$ and are adapted to this mixed structure. These estimates restore a smooth variational framework on $H^1(\mathbb R^N)$ and provide the quantitative control required to pass from the approximating problems to the original logarithmic equation.

We next investigate the autonomous limiting problem associated with the behavior of the potential at infinity. Unlike the first-order autonomous logarithmic equation, whose positive ground state is explicitly known, the present limiting problem contains both the first-order and the higher-order logarithmic terms and does not admit an explicit ground-state formula. We therefore analyze the structure of the limiting nonlinearity, combine the strong-maximum-principle/compact-support dichotomy with the corresponding symmetry theory, and verify a Serrin-Tang type condition \cite{Serrin2000}. This yields uniqueness, up to translations, of the nonnegative autonomous profile with connected positivity set and, in particular, the no-critical-value property below the two-profile threshold. This property is essential for identifying the nontrivial profiles that may arise when Palais-Smale sequences lose compactness at infinity.

Finally, we establish a profile decomposition for Palais-Smale sequences of the approximating problems. Although the concentration-compactness and barycenter ideas are motivated by the strategy of \cite{fengwx2024}, the presence of the higher-order logarithmic term requires a new quantitative control of the energy level. More precisely, we construct a barycenter-based min--max level and derive an upper bound that is strictly below twice the least energy of a nontrivial solution of the autonomous limiting problem. If $c$ denotes this least energy, then the resulting critical level satisfies $c<\bar c_p<2c$. Since every nontrivial profile carries at least the energy $c$, the strict upper bound $\bar c_p<2c$ excludes the occurrence of two or more nontrivial profiles. This prevents energy splitting at the relevant level and restores the compactness needed to pass to the limit and obtain a solution of the original logarithmic problem.

The main purpose of this paper is to investigate the existence of nontrivial nonnegative bound-state solutions to equation (\ref{l1: eq1-1}). Specifically, the main theorem is as follows:
\begin{theorem}
	\label{l1: sec1-1}
Assume $\alpha,\beta,\gamma>0$, $m\in\mathbb N$ is odd, and the following conditions hold.
	
$(AC_1)$ $V\in C(\mathbb R^N)$ satisfies $0<V_0\le V(x)\le V^*<1+\ln2\;\;\text{for all }x\in\mathbb R^N$, where $V_0:=\inf\limits_{x\in\mathbb R^N}V(x),\;\;V^*:=\sup\limits_{x\in\mathbb R^N}V(x)$, and $\lim\limits_{|x|\to\infty}V(x)=1$. 
	
$(AC_2)$ Let $\ell_0$ be the unique real number satisfying $\gamma\ell_0^m+\beta\ell_0-1=0$. Assume that, for every $\ell>\ell_0$,
$$
\bigl(m\gamma\ell^{m-1}+\beta\bigr)^2
\ge m(m-1)\gamma\ell^{m-2}
\bigl(\gamma\ell^m+\beta\ell-1\bigr).
$$
	
$(AC_3)$ Let $w\in\mathcal D$ be a nonnegative ground state of the autonomous limit problem \eqref{l1: eq2-3}. Set
$$
C(\beta,m,w):=\sup_{0<a\leq 2^{\frac{1}{\beta}}}
\frac{\displaystyle\int_{\R^N}\int_0^a
	\left(\ln|\tau|+\ln|w(x)|\right)^{m-1}w^2(x)\tau\,d\tau dx}
{\displaystyle a\int_{\R^N}\int_0^1
	\left(\ln|\tau|+\ln|w(x)|\right)^{m-1}w^2(x)\tau\,d\tau dx}.
$$
At points where $w(x)=0$, the integrands involving $w^2(x)(\ln w(x))^j$ are understood by continuous extension and are set equal to zero. Assume that
\[
\Lambda_0:=\max\bigl\{2^{\frac{2}{\beta}},\,\;C(\beta,m,w)2^{\frac{1}{\beta}}\bigr\}<2.
\]
Then problem (\ref{l1: eq1-1}) has a nontrivial nonnegative bound state solution. Moreover, when $m=1$, this solution is positive in $\R^N$.
\end{theorem}

\begin{remark}\label{l1: sec1-2}
(i) The restriction to odd $m$ is structural.  It gives $\Phi'(s)=\beta+m\gamma s^{m-1}\ge\beta>0$ for $\Phi(s)=\beta s+\gamma s^m$, and it also fixes the sign of the small-amplitude logarithmic part used in the splitting argument.

(ii)  Condition $(AC_1)$ is a restriction on the potential function; its primary purpose is to allow us to derive the desired boundedness results through appropriate scaling in the proofs of \Cref{l1: sec2-5,l1: sec2-6,l1: sec2-7}. Moreover, the additional bound $V^*<1+\ln2$ is used in \Cref{l1: sec3-5} to obtain a uniform estimate for the Nehari scaling parameter.

(iii) Condition $(AC_2)$ is an explicit structural condition on the autonomous nonlinearity.  Writing
\[
h(\ell)=\gamma\ell^m+\beta\ell-1,
\]
it is exactly the inequality $h'(\ell)^2\ge h''(\ell)h(\ell)$ on $(\ell_0,\infty)$.  In \Cref{l1: sec2-1} this condition yields the Serrin-Tang monotonicity criterion and, together with the symmetry and compact-support properties of nonnegative autonomous solutions, the uniqueness up to translations of the autonomous solution with connected positivity set; the no-critical-value property in $(c_\infty,2c_\infty)$ is then a conclusion rather than an additional hypothesis.

(iv) The additional requirement $\Lambda_0:=\max\bigl\{2^{\frac{2}{\beta}},\,\;C(\beta,m,w)2^{\frac{1}{\beta}}\bigr\}<2$ in condition $(AC_3)$ guarantees the upper bound $\bar c<2c$ which is established in Section 4. This energy inequality places the Palais-Smale level below twice the ground state energy level, thereby preventing the Palais-Smale sequence from decomposing into multiple nontrivial profiles.

(v) The result can be regarded as a higher-order extension of the logarithmic Schr\"odinger equation considered in \cite{fengwx2024}. Although the min-max construction follows the high-energy scheme developed there, the higher-order logarithmic nonlinearity requires additional estimates, especially in the uniform bounds and profile decomposition arguments.

(vi) The present model extends the framework in \cite{anx2025} by incorporating both the potential function term $V(x)u$ and the first-order logarithmic nonlinearity term $u\ln u^2$. The restriction that $m$ is an integer is mainly due to the integration by parts arguments used in the proof. Whether similar results hold for non-integer orders remains an interesting problem.
\end{remark}

The article is organized as follows. In Section 2, we introduce the variational framework for the approximating problem and study the corresponding autonomous limit problem. In particular, we establish the uniqueness, up to translations, of the nonnegative autonomous profile with connected positivity set and derive several properties that are essential for the subsequent compactness analysis. In Section 3, we establish uniform estimates with respect to the approximation parameter and investigate the profile decomposition of Palais-Smale sequences. These results allow us to overcome the loss of compactness caused by the unbounded domain. Finally, in Section 4, we construct a suitable min--max level by using the barycenter technique and prove the existence of a nontrivial nonnegative bound state for the logarithmic Schr\"odinger equation.

\vskip2mm
{\section{The variational setting and basic properties}}
\setcounter{equation}{0}
\noindent
In this section, we introduce the variational framework for the approximating problem and recall some basic properties of the associated autonomous problem. Since the logarithmic nonlinearity prevents the direct application of the standard variational theory, we first consider a family of smooth approximating problems. The properties established in this section will play an essential role in the construction of critical points and the compactness analysis in the subsequent sections.
\subsection{The logarithmic problem and its approximation}
\noindent
To overcome the lack of smoothness caused by the logarithmic nonlinearity, we introduce a family of power-type approximating problems. For each $p>2$ sufficiently close to $2$, the corresponding energy functional is of class $C^1$ on $H^1(\mathbb R^N)$, which enables us to employ the standard variational framework. Now, we define the energy functional $I:  H^1(\mathbb{R}^N)\rightarrow \mathbb{R}\cup\{+\infty\}$ of equation (\ref{l1: eq1-1}) as
\begin{small}
\begin{eqnarray*}
I(u)
= \frac{\alpha}{2}\int_{\R^N}|\nabla u|^2dx+\frac{1}{2}\int_{\R^N}V(x)u^2dx
+\frac{\beta}{4}\int_{\R^N}u^2dx-\frac{\beta}{2}\int_{\R^N}u^2\ln |u| dx
-\int_{\R^N}\frac{\gamma u^2}{2}\sum_{k=0}^{m}(-1)^k\frac{m!}{(m-k)!}\frac{(\ln |u|)^{m-k}}{2^k}dx.
\end{eqnarray*}
\end{small}$H^1(\R^N)$ is a Hilbert space with the following inner product and norm
$$
\langle u,v \rangle:= \int_{\R^N}\nabla u\cdot \nabla v+u\cdot vdx,\;\;\;\;\|u\|:=\sqrt{\langle u, u \rangle},\;\;\;\; \forall u,v \in H^1(\R^N).
$$
For convenience, we denote 
\begin{eqnarray}
	\label{l1: eq2-1}
F(u)=\frac{u^2}{2}\sum_{k=0}^{m}(-1)^k\frac{m!}{(m-k)!}\frac{(\ln |u|)^{m-k}}{2^k}:=\int_0^u(\ln |t|)^mtdt.
\end{eqnarray}
We also set
$$
G(u):=\int_0^u(\ln |t|)^{m-1}tdt.
$$
Consequently, the energy functional $I$ becomes
$$
I(u)
= \frac{\alpha}{2}\int_{\R^N}|\nabla u|^2dx+\frac{1}{2}\int_{\R^N}V(x)u^2dx+\frac{\beta}{4}\int_{\R^N}u^2dx-\frac{\beta}{2}\int_{\R^N}u^2\ln |u| dx-\gamma\int_{\R^N} F(u)dx.
$$

The logarithmic energy is not continuous on the whole space $H^1(\mathbb R^N)$.  We therefore work in
\[
\mathcal D:=\left\{u\in H^1(\mathbb R^N):
\int_{\mathbb R^N}|u|^2|\ln|u||^m\,dx<\infty\right\}.
\]
Since $m\ge1$, the $m$-th logarithmic moment controls the lower-order logarithmic moment near the origin, while the part $|u|>1$ is controlled by a subcritical Sobolev power.  Thus both nonlinear terms in (\ref{l1: eq2-2}) are meaningful on $\mathcal D$.  In particular, the estimate used in Proposition~2.3 of \cite{anx2025} gives that
\[
\int_{\mathbb R^N}uv(\ln|u|)^m\,dx
\]
is well defined for $u,v\in\mathcal D$; the same is then true for the first logarithmic power.

For the weak formulation, we test only with $\varphi\in C_c^\infty(\mathbb R^N)$ and set
\begin{eqnarray}
\label{l1: eq2-2}
\langle I'(u),\varphi\rangle
=\alpha\int \nabla u\cdot\nabla\varphi
+\int V(x)u\varphi
-\beta\int u\varphi\ln|u|
-\gamma\int u\varphi(\ln|u|)^m.
\end{eqnarray}
A critical point means that this quantity is zero for every $\varphi\in C_c^\infty(\mathbb R^N)$.

The limit problem corresponding to the equation (\ref{l1: eq1-1}) is represented as
\begin{eqnarray}
\label{l1: eq2-3}
- \alpha \Delta u +u -\beta u\ln |u|=\gamma  u (\ln |u|)^m,\;\;u\in  H^1(\mathbb{R}^N),
\end{eqnarray}
and the energy functional $I_\infty: H^1(\mathbb{R}^N)\rightarrow \mathbb{R}\cup\{+\infty\}$ of equation (\ref{l1: eq2-3}) writes
\begin{eqnarray*}
I_\infty(u)
=\frac{\alpha}{2}\int_{\R^N}|\nabla u|^2dx+\frac{1}{2}\int_{\R^N}u^2dx
+\frac{\beta}{4}\int_{\R^N}u^2dx-\frac{\beta}{2}\int_{\R^N}u^2\ln |u| dx-\gamma\int_{\R^N} F(u)dx,
\end{eqnarray*}
and one has
$$
\langle I_\infty'(u),v \rangle
=\alpha\int_{\mathbb{R}^N} \nabla u \cdot \nabla vdx
+\int_{\mathbb{R}^N}u\cdot vdx-\beta\int_{\mathbb{R}^N} u\cdot v\ln{|u|}dx-\gamma\int_{\mathbb{R}^N} u\cdot v(\ln{|u|})^mdx.
$$

Motivated by \cite{anx2025,wangzq2019}, we consider a power-law approximation of (\ref{l1: eq1-1}) as follows:
\begin{eqnarray}
\label{l1: eq2-4}
-\alpha \Delta u+V(x)u
-\beta\frac{|u|^{p-2}-1}{p-2}u
=\gamma\left(\frac{|u|^{p-2}-1}{p-2}\right)^m u.
\end{eqnarray}
For $N\geq 1$, fix $q\in(2,2^*)$ and choose $\bar p>2$ such that $2<p<\bar p\leq 2+\frac{q-2}{m}$. Here, the condition $2<p<\bar p$ specifies the admissible range of the perturbation parameter $p$, whereas the upper bound $\bar p\leq 2+\frac{q-2}{m}$, equivalently $m(\bar p-2)\leq q-2$, is imposed to obtain the uniform growth estimate in Lemma~2.3. Indeed, this condition allows the highest-order $p$-dependent growth to be controlled uniformly, for all $2<p\leq\bar p$, by the fixed power $s^{q-2}$. Such a uniform subcritical control is essential, since the highest-order term in (\ref{l1: eq2-4}) exhibits the power growth $|u|^{1+m(p-2)}$, and therefore the perturbed functional is not automatically well defined for every $p\in(2,2^*)$ when $m>1$.

The energy functional associated with (\ref{l1: eq2-4}) is
\begin{small}
\begin{eqnarray*}
I_p(u)
=\frac{\alpha}{2}\int_{\R^N}|\nabla u|^2dx
+\frac12\int_{\R^N}V(x)u^2dx
+\frac{\beta}{2p}\int_{\R^N}|u|^pdx
-\frac{\beta}{2}\int_{\R^N}\frac{|u|^p-u^2}{p-2}dx
-\gamma\int_{\R^N}\int_0^u
\left(\frac{|t|^{p-2}-1}{p-2}\right)^m tdtdx.
\end{eqnarray*}
\end{small}It is $C^1$ on $H^1(\mathbb R^N)$ for $2<p<\bar{p}$. The term $\gamma\displaystyle\int_{\R^N} \displaystyle\int_0^u \left(\frac{|t|^{p-2}-1}{p-2}\right)^m tdtdx$ represents the integration of the primitive function of the higher-order term in (\ref{l1: eq2-4}). It follows from the integration by parts that 
$$
\gamma\int_{\R^N} \int_0^u \left(\frac{|t|^{p-2}-1}{p-2}\right)^m tdtdx=\frac{\gamma}{2}\int_{\R^N}u^2\left(\frac{|u|^{p-2}-1}{p-2}\right)^mdx-\frac{m\gamma}{2}\int_{\R^N}\int_0^u\left(\frac{|t|^{p-2}-1}{p-2}\right)^{m-1}|t|^{p-2}tdtdx.
$$
For every $s>0$,
\[
\frac{s^{p-2}-1}{p-2}\longrightarrow\ln s
\qquad\text{as }p\to2^+.
\]
We shall let $p\to2^+$ only through the admissible interval above.

\subsection{Nehari manifold and variational characterization}
\noindent
The Nehari manifold provides a natural constraint for the variational analysis of the approximating problem. In this subsection, we recall the characterization of the least energy levels through the Nehari manifold and establish several relations between the variational levels of the approximating and limiting problems.

To investigate the existence of bound state solution to equation (\ref{l1: eq1-1}), we define
\begin{eqnarray*}
\varphi_p(u):=\langle I'_p(u),u\rangle
=  \int_{\R^N}\alpha |\nabla u|^2dx+\int_{\R^N}V(x)u^2dx-\beta\int_{\R^N}\frac{|u|^{p}-u^2}{p-2} dx-\gamma\int_{\R^N} \left(\frac{|u|^{p-2}-1}{p-2}\right)^m u^2dx
\end{eqnarray*}
and the Nehari manifold of $I_p$
$$
\mathcal{N}_p:=\Big\{u\in  H^1(\mathbb{R}^N)\setminus\{0\}\mid \varphi_p(u)=0\Big\}\neq \emptyset.
$$

The tangent space of $\mathcal N_p$ at $u$ is
\[
T_u\mathcal N_p:=\{v\in H^1(\mathbb R^N):\langle\varphi_p'(u),v\rangle=0\}.
\]
The norm of the derivative of the restriction is
\[
\|I_p'|_{\mathcal N_p}(u)\|_*
:=\sup_{\substack{v\in T_u\mathcal N_p,\\\|v\|=1}}\langle I_p'(u),v\rangle
=\min_{\lambda\in\mathbb R}\|I_p'(u)-\lambda\varphi_p'(u)\|_{H^{-1}}.
\]
Set $\varphi(u)=\langle I'(u),u\rangle$ and
$\varphi_\infty(u)=\langle I_\infty'(u),u\rangle$, and define
\[
\mathcal N=\{u\in\mathcal D\setminus\{0\}:\varphi(u)=0\},\quad
\mathcal N_\infty=\{u\in\mathcal D\setminus\{0\}:\varphi_\infty(u)=0\},
\]
\[
c_p=\inf_{u\in\mathcal N_p}I_p(u),\qquad c=\inf_{u\in\mathcal N}I(u),\qquad c_\infty=\inf_{u\in\mathcal N_\infty}I_\infty(u).
\]
For every $u\ne0$ the function
\[
t\longmapsto\frac{\varphi_p(tu)}{t^2}
=\int_{\mathbb R^N}(\alpha|\nabla u|^2+Vu^2)dx-\int_{\mathbb R^N}\bigl[\beta \frac{(t|u|)^{p-2}-1}{p-2}+\gamma \left(\frac{(t|u|)^{p-2}-1}{p-2}\right)^m\bigr]u^2dx,
\]
is strictly decreasing, because, for every fixed $s>0$, the map $t\mapsto\frac{(t|u|)^{p-2}-1}{p-2}$ is strictly increasing, and the map $(\cdot)\mapsto\beta (\cdot)+\gamma (\cdot)^m$ is strictly increasing.  Hence each ray meets $\mathcal N_p$ exactly once.  Consequently
\[
c_p=\inf_{u\ne0}\max_{t>0}I_p(tu).
\]
The same monotonicity argument gives the corresponding ray characterizations for $c$ and $c_\infty$.  Moreover, for every $u\in\mathcal N_p$,
\[
-\langle\varphi_p'(u),u\rangle
=\int_{\mathbb R^N}\bigl[\beta+m\gamma \left(\frac{|u|^{p-2}-1}{p-2}\right)^{m-1}\bigr]|u|^pdx>0.
\]
In particular, $\varphi_p'(u)\neq0$ for every $u\in\mathcal N_p$.

\subsection{Properties of the autonomous limit problem}
\noindent
Next, we investigate the qualitative properties of the autonomous limit problem. These properties are fundamental for the later compactness analysis, since every nontrivial profile arising from the loss of compactness is related to a solution of the autonomous equation. In particular, the uniqueness up to translations of the nonnegative autonomous profile with connected positivity set plays a crucial role in identifying the limiting profiles.
\begin{lemma}\label{l1: sec2-1}
	\it Assume $(AC_2)$ holds. Let $w\in\mathcal D$ be a nonnegative ground state of the autonomous limit problem \eqref{l1: eq2-3}. Then every nonnegative solution of \eqref{l1: eq2-3} whose positivity set is connected is unique up to translations. Consequently, the autonomous functional $I_\infty$ admits no critical value in $(c_\infty,2c_\infty)$.
\end{lemma}

\noindent
{\bf Proof.}
For $s>0$, write \eqref{l1: eq2-3} as
\[
-\Delta u=f(u),\qquad f(s)=\frac1\alpha s\left[\gamma(\ln s)^m+\beta\ln s-1\right].
\]
Set
\[
h(\ell)=\gamma\ell^m+\beta\ell-1,\qquad \ell=\ln s.
\]
Since $m$ is odd,
\[
h'(\ell)=m\gamma\ell^{m-1}+\beta>0\qquad\text{for all }\ell\in\R.
\]
Hence there exists a unique $\ell_0\in\R$ such that $h(\ell_0)=0$. Moreover, since $h(0)=-1<0$ and $h(\ell)\to+\infty$ as $\ell\to+\infty$, the monotonicity of $h$ yields $\ell_0>0$. Let
\[
b=e^{\ell_0}.
\]
Then
\[
f(s)\leq0\quad(0<s\leq b),\qquad f(s)>0\quad(s>b).
\]
Since $s|\ln s|^j\to0$ as $s\to0^+$ for every fixed $j\geq1$, we extend $f$ continuously by setting $f(0)=0$.

We first describe the behavior of a nonnegative solution near zero. Let
\[
F(s)=\int_0^s f(t)\,dt.
\]
As $s\to0^+$, since $m$ is odd,
\[
F(s) = -\frac{\gamma}{2\alpha} s^2|\ln s|^m(1+o(1)).
\]
Consequently,
\[
\int_{0^+}\frac{ds}{\sqrt{|F(s)|}}\sim\sqrt{\frac{2\alpha}{\gamma}}\int_{0^+}\frac{ds}{s|\ln s|^{m/2}}.
\]
Thus,
\[
\int_{0^+}\frac{ds}{\sqrt{|F(s)|}}=
\begin{cases}
	+\infty, & m=1,\\
	<+\infty, & m\geq3.
\end{cases}
\]
Moreover,
\[
f'(s)=\frac1\alpha\left[h(\ln s)+h'(\ln s)\right]<0
\]
for $s>0$ sufficiently small. Therefore, by the strong maximum principle and compact-support principle for this class of nonlinearities \cite{Pucci1999,Gazzola2000}, a nonnegative bound-state solution is positive in $\R^N$ when $m=1$, while for $m\geq3$ it has compact support. In either case, if its positivity set is connected, the
symmetry result in \cite{Serrin1999} shows that, up to a translation, the solution is radial and strictly decreasing in its positivity set.

Thus, it remains to show uniqueness among such radial solutions. For $s>b$, define
\[
K(s)=\frac{sf'(s)}{f(s)}.
\]
Since
\[
f(s)=\alpha^{-1}sh(\ln s),
\]
we have
\[
K(s)=1+\frac{h'(\ell)}{h(\ell)},
\qquad \ell=\ln s,
\]
and therefore
\[
\frac{d}{d\ell}
\left(
\frac{h'(\ell)}{h(\ell)}
\right)
=
\frac{h''(\ell)h(\ell)-h'(\ell)^2}{h(\ell)^2}
\leq0,
\qquad \ell>\ell_0,
\]
by $(AC_2)$. Hence $K$ is non-increasing on $(b,+\infty)$.

For $N\geq3$, the Serrin-Tang uniqueness criterion \cite{Serrin2000} applies and yields the uniqueness of the radial
solution. For $N=2$, we use the low-dimensional extension of the Serrin-Tang criterion in \cite{Pucci2025}. Indeed, besides the above monotonicity of $K$, we have
\[
K(s)
=
1+\frac{h'(\ln s)}{h(\ln s)}
>1>-1,
\qquad s>b,
\]
so that the additional low-dimensional condition is satisfied. Therefore, the radial solution is also unique when $N=2$.

For $N=1$, the same conclusion follows directly from the first integral. Since $f<0$ on $(0,b)$ and $f>0$ on $(b,\infty)$,
$F$ is strictly decreasing on $(0,b)$ and strictly increasing on $(b,\infty)$. Moreover,
\[
F(b)<0,\qquad F(s)\to+\infty
\quad\text{as }s\to+\infty.
\]
Hence there exists a unique $b^*>b$ satisfying
\[
F(b^*)=0.
\]
Let $u$ be a nonnegative solution whose positivity set is connected. After a translation, we may assume that $u$ attains its maximum at $x=0$. Multiplying the equation by $u'$, we obtain
\[
\frac12|u'|^2+F(u)=0
\]
on the positivity interval of $u$. Since $u'(0)=0$, it follows that
\[
F(u(0))=0.
\]
By the uniqueness of $b^*>b$ satisfying $F(b^*)=0$, we have
\[
u(0)=b^*.
\]
Then $u$ is uniquely determined on its positivity interval by
\[
u(0)=b^*,\qquad u'(0)=0.
\]
Consequently, in any dimension, every nonnegative solution of \eqref{l1: eq2-3} whose positivity set is connected is unique up to
translations.

For $m\geq3$, the compact-support/free-boundary theory cited above gives the Dirichlet--Neumann free-boundary condition on each positivity component. Hence the restriction of a compactly supported nonnegative solution to any positivity component, extended by zero, is again a weak solution of the autonomous equation.

We next observe that the ground state $w$ has connected positivity set. Indeed, if $\{w>0\}$ had at least two connected components, then the restrictions of $w$ to two such components, extended by zero, would give two nontrivial autonomous solutions
$w_1,w_2\in\mathcal N_\infty$. Hence
\[
I_\infty(w)
\geq
I_\infty(w_1)+I_\infty(w_2)
\geq2c_\infty,
\]
which contradicts
\[
I_\infty(w)=c_\infty>0.
\]
Thus $\{w>0\}$ is connected, and the uniqueness proved above applies to $w$.

Finally, let $Z\in\mathcal D\setminus\{0\}$ be a critical point of $I_\infty$. If $Z$ is sign-changing, write
\[
Z^+=\max\{Z,0\},
\qquad
Z^-=\max\{-Z,0\}.
\]
By the same cutoff and self-testing argument used for logarithmic critical points, testing the equation by $Z^+$ and $-Z^-$ gives
\[
\langle I'_\infty(Z^+),Z^+\rangle=0,
\qquad
\langle I'_\infty(Z^-),Z^-\rangle=0.
\]
Hence
\[
Z^+,Z^-\in\mathcal N_\infty.
\]
It follows from the definition of $I_\infty$ that
\[
I_\infty(Z)
=
I_\infty(Z^+)+I_\infty(Z^-)
\geq2c_\infty.
\]
Thus, a critical point with energy below $2c_\infty$ cannot be sign-changing.

Suppose now that $Z$ is of one sign. Replacing $Z$ by $|Z|$ if necessary, we may assume that $Z\geq0$. If $m=1$, the above strong
maximum principle gives $Z>0$, and hence its positivity set is connected. If $m\geq3$, the compact-support principle gives that $Z$ has compact support. If $\{Z>0\}$ had at least two connected components, the corresponding restrictions $Z_1$ and $Z_2$, extended by zero, would satisfy
\[
Z_1,Z_2\in\mathcal N_\infty,
\]
and consequently
\[
I_\infty(Z)
\geq
I_\infty(Z_1)+I_\infty(Z_2)
\geq2c_\infty.
\]
Therefore, whenever
\[
I_\infty(Z)<2c_\infty,
\]
the positivity set of $Z$ must be connected.

By the uniqueness proved above, there exists $y\in\R^N$ such that
\[
Z=w(x-y).
\]
Hence
\[
I_\infty(Z)=c_\infty.
\]
Consequently, $I_\infty$ admits no critical value in $(c_\infty,2c_\infty)$. The proof is completed.\qed

\begin{remark}\label{l1: sec2-2}
	Condition $(AC_2)$ is mainly used to verify the monotonicity condition in the Serrin-Tang type uniqueness criterion. Together with the symmetry and compact-support properties of nonnegative autonomous solutions, this condition yields the uniqueness, up to translations, of the nonnegative autonomous solution with connected positivity set. Consequently, the no-critical-value property in $(c_\infty,2c_\infty)$ follows.
\end{remark}

\section{Uniform estimates and compactness analysis}
\noindent
In this section, we establish several uniform estimates for the approximating problems and analyze the compactness properties of their Palais--Smale sequences. These estimates are independent of the parameter $p$ near $2$ and are crucial for passing to the logarithmic limit.
\subsection{Uniform estimates for approximation}
\begin{lemma}\label{l1: sec2-3}
\it Fix $q>2$ and choose $\bar{p}>2$ such that $\bar{p}\leq 2+\frac{q-2}{m}$. Then there exists a constant $C_{\bar{p},m}>0$ such that for all $2<p\le \bar{p}$ and $s>0$ the following estimates hold:
\[
\frac{s^{p-2}-1}{p-2}\le C_{\bar{p},m}s^{q-2}\;\;\;\;\text{and}\;\;\;\;\left(\frac{s^{p-2}-1}{p-2}\right)^m\le C_{\bar{p},m}s^{q-2}.
\]
\end{lemma}
\noindent
{\bf Proof.}
Let $2<p\le \bar{p}$. We first consider $s>1$. Denote
\[
f(t)=\frac{s^t-1}{t}.
\]
By differentiating this function with respect to the variable $t$, we can see that, for fixed $s>1$, it is an increasing function for $t>0$. By the choice of $\bar{p}$, one has
\[
m(\bar{p}-2)\leq q-2.
\]
Hence, for any $2<p\le \bar{p}$, according to the definition of $m$, we have
\[
0<\frac{s^{p-2}-1}{p-2}\le\frac{s^{\bar{p}-2}-1}{\bar{p}-2}\le\frac{1}{\bar{p}-2}s^{\bar{p}-2}
\le\frac{1}{\bar{p}-2}s^{m(\bar{p}-2)}\le\frac{1}{\bar{p}-2}s^{q-2}.
\]
Therefore,
\[
\left(\frac{s^{p-2}-1}{p-2}\right)^m
\le\frac{1}{(\bar{p}-2)^m}s^{m(\bar{p}-2)}
\le\frac{1}{(\bar{p}-2)^m}s^{q-2}.
\]
Thus, letting
\[
C_{\bar{p},m}
=\max\left\{\frac{1}{\bar{p}-2},
\frac{1}{(\bar{p}-2)^m}\right\},
\]
for $s>1$, we have
\[
\frac{s^{p-2}-1}{p-2}\le C_{\bar{p},m}s^{q-2},
\qquad
\left(\frac{s^{p-2}-1}{p-2}\right)^m
\le C_{\bar{p},m}s^{q-2}.
\]
Finally, if $0<s\le1$, then $s^{p-2}\le1$, so
\[
\frac{s^{p-2}-1}{p-2}\le0.
\]
Since $m$ is odd, we also have
\[
\left(\frac{s^{p-2}-1}{p-2}\right)^m\le0.
\]
Therefore, both estimates hold trivially for $0<s\le1$. This completes the proof.\qed

\begin{remark}\label{l1: sec2-4}
\Cref{l1: sec2-3} may be regarded as a higher-order extension of the corresponding estimate in \Cref{l1: sec2-1} of~\cite{wangzq2019}. In the present paper, $m$ is assumed to be an odd integer, and hence the estimate can be stated directly in terms of $\left(\frac{s^{p-2}-1}{p-2}\right)^m$. However, the argument above also suggests a natural extension in which $m$ is no longer restricted to odd integers. More precisely, for any real $m\geq1$, one may instead consider $\left[\frac{s^{p-2}-1}{p-2}\right]_+^m$. Indeed, when $s>1$, we have $\left[\frac{s^{p-2}-1}{p-2}\right]_+^m=\left(\frac{s^{p-2}-1}{p-2}\right)^m$, and hence the proof of \Cref{l1: sec2-3} remains unchanged. On the other hand, when $0<s\leq1$, $\left[\frac{s^{p-2}-1}{p-2}\right]_+^m=0$, so no additional estimate is required on the small-amplitude region. Therefore, under the same condition, the estimates hold for every real $m\geq1$. Consequently, we conjecture that this method of taking the positive part may eliminate the restriction on $m$ being an odd integer when estimating the growth of a function, and may be of practical value for generalizing the current high-order perturbation to a more general exponent $m\geq 1$.
\end{remark}

\begin{lemma}\label{l1: sec2-5}
\it There exists $p_0\in(2,\bar{p})$ such that for any $d_1,d_2>0$, if $u\in H^1(\mathbb{R}^N)$ and $2<p<p_0$ satisfy
\[
I_p(u)\leq d_1,\qquad\|I'_p(u)\|_{H^{-1}(\mathbb{R}^N)}\leq d_2,
\]
then there exists a constant $C(d_1,d_2)>0$ such that
$$
\|u\|_{H^1(\R^N)},\;\;\int_{\mathbb{R}^N}\Big|\frac{|u|^p-u^2}{p-2}\Big|dx,
\;\;\int_{\R^N}\int_0^u\left(\frac{|t|^{p-2}-1}{p-2}\right)^{m-1}|t|^{p-2}tdtdx,
\;\;\int_{\R^N}|u|^2\left|\frac{|u|^{p-2}-1}{p-2}\right|^mdx\leq C(d_1,d_2).
$$
Moreover, it was also found that $\int_{\R^N}|u|^pdx$ can be controlled by this constant, i.e. $\int_{\R^N}|u|^pdx\leq C(d_1,d_2)$.
\end{lemma}

\noindent
{\bf Proof.}
Since $m$ is odd, $m-1$ is even, and hence $\left(\frac{|t|^{p-2}-1}{p-2}\right)^{m-1}|t|^{p-2}t$ is odd. Moreover, for all $u\in\mathbb R$, there exists
\[
\int_0^u
\left(\frac{|t|^{p-2}-1}{p-2}\right)^{m-1}
|t|^{p-2}t\,dt\ge0.
\]
In virtue of the definition of $\|I'_p(u)\|_{H^{-1}(\mathbb{R}^N)}$ and the problem conditions, a direct computation gives
\begin{eqnarray}
	\label{l1: eq2-5}
	2d_1+d_2\|u\|_{H^1{(\R^N)}}
	&\geq& 2I_p(u)-\varphi_p(u)
	=   \frac{\beta}{p}\int_{\R^N}|u|^pdx
	+m\gamma\int_{\R^N}\int_0^u\left(\frac{|t|^{p-2}-1}{p-2}\right)^{m-1}|t|^{p-2}tdtdx.
\end{eqnarray}

We now prove that $\|u\|_{H^1(\mathbb{R}^N)}$ is bounded. Fix $S\in(2,2^*)$, where $2^*$ is the Sobolev critical exponent. Let $q_0\in(2,S)$ and $p_0\in(2,\bar p)$ are sufficiently close to 2 such that
\[
p_0<q_0,\qquad m(p_0-2)\leq q_0-2.
\]
For every $2<p<p_0$, according to \Cref{l1: sec2-3}, for $s>1$, one has
\[
0<\frac{s^{p-2}-1}{p-2}\le C_{\bar{p},m}s^{q_0-2},
\qquad
0<\left(\frac{s^{p-2}-1}{p-2}\right)^m
\le C_{\bar{p},m}s^{q_0-2};
\]
for $0<s\leq1$, we have
\[
\frac{s^{p-2}-1}{p-2}\leq0,\qquad\left(\frac{s^{p-2}-1}{p-2}\right)^m\leq0.
\]
Consequently, using $V(x)\geq V_0>0$, from the definition of $I_p$ we obtain
\begin{eqnarray}
\label{l1: eq2-6}
     d_1\geq I_p(u)
&\geq& \frac12\min\{\alpha,V_0\}\|u\|_{H^1(\mathbb{R}^N)}^2
-\frac{\beta}{2} \int_{\mathbb{R}^N}\frac{|u|^{p}-u^2}{p-2} dx
-\gamma\int_{\mathbb{R}^N} \int_0^u \left(\frac{|t|^{p-2}-1}{p-2}\right)^m tdtdx\nonumber\\
&\geq& \frac12\min\{\alpha,V_0\}\|u\|_{H^1(\mathbb{R}^N)}^2
-\frac{\beta}{2} \int_{\mathbb{R}^N}C_{\bar{p},m}|u|^{q_0}dx
-\gamma \int_{\mathbb{R}^N}\int_0^uC_{\bar{p},m}|t|^{q_0-2}tdtdx\nonumber\\
& =  & \frac12\min\{\alpha,V_0\}\|u\|_{H^1(\mathbb{R}^N)}^2
-C_{\bar{p},m}\left(\frac{\beta}{2} +\frac{\gamma}{q_0}\right)\int_{\mathbb{R}^N}|u|^{q_0}dx.
\end{eqnarray}

Since $p<q_0<S$, choose $\theta\in(0,1)$ such that
\[
\frac1{q_0}=\frac{\theta}{p}+\frac{1-\theta}{S}.
\]
Therefore, we can conclude that $\frac{q_0\theta}{p}=1-\frac{(1-\theta)q_0}{S}\in (0,1)$. By the H\"{o}lder inequality, (\ref{l1: eq2-5}), the Sobolev embedding inequality and $C_p$ inequality, we have
\begin{eqnarray*}
       \int_{\mathbb{R}^N}|u|^{q_0}dx=\int_{\mathbb{R}^N}|u|^{\theta q_0}|u|^{(1-\theta) q_0}dx
& =  & \int_{\mathbb{R}^N}(|u|^p)^{\frac{\theta q_0}{p}}(|u|^S)^{\frac{(1-\theta)q_0}{S}}dx\\
&\leq& \left[\int_{\mathbb{R}^N}\left((|u|^p)^{\frac{\theta q_0}{p}}\right)^{\frac{p}{\theta q_0}}dx\right]^{\frac{\theta q_0}{p}}
\left[\int_{\mathbb{R}^N}\left((|u|^S)^{\frac{(1-\theta)q_0}{S}}\right)^{\frac{S}{(1-\theta) q_0}}dx\right]^{\frac{(1-\theta) q_0}{S}}\\
& =  & \left(\int_{\mathbb{R}^N}|u|^pdx\right)^{\frac{\theta q_0}{p}}
\left(\int_{\mathbb{R}^N}|u|^Sdx\right)^{\frac{(1-\theta) q_0}{S}}\\
&\leq& \left[\frac{p}{\beta}\big(2d_1+d_2\|u\|_{H^1{(\R^N)}}\big)\right]^{\frac{\theta q_0}{p}}
\left(C_1\|u\|_{H^1{(\R^N)}}\right)^{(1-\theta) q_0}\\
& =  & \left(\frac{2pd_1}{\beta}\right)^{\frac{\theta q_0}{p}}C_1^{(1-\theta) q_0}\|u\|_{H^1{(\R^N)}}^{(1-\theta) q_0}
+\left(\frac{pd_2}{\beta}\right)^{\frac{\theta q_0}{p}}C_1^{(1-\theta) q_0}\|u\|_{H^1{(\R^N)}}^{\frac{\theta q_0}{p}+(1-\theta) q_0}.
\end{eqnarray*}
Based on the initial range of parameter selections, it can be seen that $\theta=\frac{p(S-q_0)}{q_0(S-p)}$. Then, when $p,q_0 \to 2^+$, we have
\begin{eqnarray*}
       \frac{\theta q_0}{p}+(1-\theta)q_0
 =\frac{q_0}{p}\times\frac{p(S-q_0)}{q_0(S-p)}+\left[1-\frac{p(S-q_0)}{q_0(S-p)}\right]q_0
& =  &\frac{S-q_0+S(q_0-p)}{S-p}\\
& \to&\frac{S-2+S(2-2)}{S-2}=1.
\end{eqnarray*}
So, we can set $\frac{\theta q_0}{p}+(1-\theta)q_0<\frac{3}{2}$. Thus, there exists
\begin{eqnarray}
\label{l1: eq2-7}
\int_{\mathbb{R}^N}|u|^{q_0}dx \leq C_2\|u\|_{H^1(\mathbb{R}^N)}^{\frac{3}{2}}.
\end{eqnarray}
Combining (\ref{l1: eq2-6}) and (\ref{l1: eq2-7}), we can obtain
\[
d_1\geq \frac12\min\{\alpha,V_0\}\|u\|_{H^1(\mathbb{R}^N)}^2
-C_2C_{\bar{p},m}\left(\frac{\beta}{2} +\frac{\gamma}{q_0}\right)\|u\|_{H^1(\mathbb{R}^N)}^{3/2},
\]
which is equivalent to $\|u\|_{H^1(\mathbb{R}^N)}\leq C(d_1,d_2)$. Substituting the results of (\ref{l1: eq2-7}) and $\|u\|_{H^1(\mathbb{R}^N)}\leq C(d_1,d_2)$ back into (\ref{l1: eq2-5}), we get
$$
\int_{\mathbb{R}^N}|u|^pdx\leq C(d_1,d_2)\qquad\text{and}\qquad
\int_{\mathbb{R}^N}\int_0^u\left(\frac{|t|^{p-2}-1}{p-2}\right)^{m-1}|t|^{p-2}t\,dt\,dx\leq C(d_1,d_2).
$$

We next estimate the higher-order term. Set
\[
A_0:=\{x\in\mathbb{R}^N:|u(x)|\leq1\},\qquad
A_1:=\{x\in\mathbb{R}^N:|u(x)|>1\}.
\]
On $A_1$, by \Cref{l1: sec2-3}, (\ref{l1: eq2-7}) and $\|u\|_{H^1(\mathbb{R}^N)}\leq C(d_1,d_2)$, one has
\[
\int_{A_1}u^2\left(\frac{|u|^{p-2}-1}{p-2}\right)^m dx
\leq \int_{\mathbb{R}^N}C_{\bar{p},m}|u|^{q_0}dx\leq C(d_1,d_2),
\]
and similarly,
\[
\int_{A_1}u^2\left(\frac{|u|^{p-2}-1}{p-2}\right) dx\leq C(d_1,d_2).
\]
On $A_0$, we have
\[
\left(\frac{|u|^{p-2}-1}{p-2}\right) \leq0,\qquad \left(\frac{|u|^{p-2}-1}{p-2}\right)^m \leq0.
\]
According to the norm of the derivative of the restriction, we have 
$$
|\varphi_p(u)|=|\langle I'_p(u),u|\leq d_2\|u\|_{H^1(\mathbb{R})^N}\leq C(d_1,d_2)
$$
Using the definition of $\varphi_p(u)$, after a few simple transformations, we can also derive
\begin{eqnarray*}
& & \varphi_p(u)+\beta\int_{A_1}u^2\left(\frac{|u|^{p-2}-1}{p-2}\right)dx
+\gamma\int_{A_1}u^2\left(\frac{|u|^{p-2}-1}{p-2}\right)^m dx\\
&=&\int_{\mathbb{R}^N}\left(\alpha|\nabla u|^2+V(x)u^2\right)dx
-\beta\int_{A_0}u^2\left(\frac{|u|^{p-2}-1}{p-2}\right)dx
-\gamma\int_{A_0}u^2\left(\frac{|u|^{p-2}-1}{p-2}\right)^m dx.
\end{eqnarray*}
Based on the above boundedness result, we can conclude that
$$
\int_{A_0}u^2\left|\frac{|u|^{p-2}-1}{p-2}\right|dx\leq C(d_1,d_2)\qquad\text{and}\qquad
\int_{A_0}u^2\left|\frac{|u|^{p-2}-1}{p-2}\right|^m dx\leq C(d_1,d_2).
$$
Hence, one has
$$
\int_{\mathbb{R}^N}\left|\frac{|u|^{p}-u^2}{p-2}\right|dx\leq C(d_1,d_2)\qquad\text{and}\qquad
\int_{\mathbb{R}^N}u^2\left|\frac{|u|^{p-2}-1}{p-2}\right|^m dx\leq C(d_1,d_2).
$$
Thus, this proof is completed.\qed

\subsection{Convergence of energy levels}
\begin{lemma}\label{l1: sec2-6}
	\it Assume the conditions $(AC_1)$ holds, there exists $0<\liminf\limits_{p\to2^+}c_p\le\limsup\limits_{p\to2^+}c_p\le c\le c_\infty$.
\end{lemma}
\noindent
{\bf Proof.}
The proof is divided into three steps.

\medskip
\noindent
{\it Step 1. We prove that $c\leq c_\infty$.}

By the definition of $c_\infty$, for every $\varepsilon>0$ there exists
$w\in\mathcal N_\infty$ such that
\[
I_\infty(w)
=\max_{t>0}I_\infty(tw)
<c_\infty+\varepsilon.
\]
Let $w_h(x)=w(x-h)$. For $t>0$, define
\[
F_h(t)
:=\int_{\mathbb R^N}\bigl(\alpha|\nabla w_h|^2+V(x)w_h^2\bigr)\,dx
-\int_{\mathbb R^N}
\bigl[\beta\ln(t|w_h|)+\gamma(\ln(t|w_h|))^m\bigr]w_h^2\,dx
\]
and
\[
F_\infty(t)
:=\int_{\mathbb R^N}\bigl(\alpha|\nabla w|^2+w^2\bigr)\,dx
-\int_{\mathbb R^N}
\bigl[\beta\ln(t|w|)+\gamma(\ln(t|w|))^m\bigr]w^2\,dx.
\]
Since $m$ is odd, the map
\[
s\longmapsto \beta s+\gamma s^m
\]
is strictly increasing on $\mathbb R$. Hence both $F_h$ and $F_\infty$ are strictly decreasing on $(0,\infty)$.
Let $t_h>0$ be the unique number such that $t_hw_h\in\mathcal N$; equivalently,
\[
F_h(t_h)=0.
\]
By $(AC_1)$ and the dominated convergence theorem,
\[
\int_{\mathbb R^N}V(x)w_h^2\,dx
=\int_{\mathbb R^N}V(y+h)w^2(y)\,dy
\longrightarrow\int_{\mathbb R^N}w^2\,dy
\qquad\text{as }|h|\to\infty.
\]
Moreover, the nonlinear terms are translation invariant. Therefore
\[
F_h(t)-F_\infty(t)
=\int_{\mathbb R^N}\bigl(V(y+h)-1\bigr)w^2(y)\,dy
\longrightarrow0
\]
as $|h|\to\infty$, uniformly for $t$ in compact subsets of $(0,\infty)$.
Since $w\in\mathcal N_\infty$, $F_\infty(1)=0$, and the zero of $F_\infty$ is unique. It follows that
\[
t_h\longrightarrow1\qquad\text{as }|h|\to\infty.
\]
Using again $(AC_1)$, the translation invariance of the nonlinear terms, and the continuity with respect to the scaling parameter, we obtain
\[
I_\infty(t_hw_h)=I_\infty(t_hw(x-h))\longrightarrow I_\infty(w)\qquad\text{as }|h|\to\infty.
\]
Since $t_hw_h\in\mathcal N$, by the definition of $c$,
\[
c\leq I(t_hw_h).
\]
Thus
\[
c\leq I_\infty(w)<c_\infty+\varepsilon.
\]
Since $\varepsilon>0$ is arbitrary, we conclude that
\[
c\leq c_\infty.
\]

\medskip
\noindent
{\it Step 2. We prove that $\limsup\limits_{p\to2^+}c_p\leq c$.}

Fix $u\in\mathcal N$. For $p>2$ sufficiently close to $2$, let $t_p>0$ be the unique number such that $t_pu\in\mathcal N_p$. Define
\[
F_p(t)
:=\int_{\mathbb R^N}\bigl(\alpha|\nabla u|^2+V(x)u^2\bigr)\,dx
-\int_{\mathbb R^N}
\left[
\beta\frac{|tu|^{p-2}-1}{p-2}
+\gamma\left(\frac{|tu|^{p-2}-1}{p-2}\right)^m
\right]u^2\,dx
\]
and
\[
F(t)
:=\int_{\mathbb R^N}\bigl(\alpha|\nabla u|^2+V(x)u^2\bigr)\,dx
-\int_{\mathbb R^N}
\bigl[\beta\ln(t|u|)+\gamma(\ln(t|u|))^m\bigr]u^2\,dx.
\]
Then $F_p(t_p)=0$ and $F(1)=0$. As above, $F_p$ and $F$ are strictly decreasing in $t$.
Furthermore, by \Cref{l1: sec2-3}, the logarithmic integrability of $u\in\mathcal D$, and the Sobolev embedding, one has
\[
F_p(t)\longrightarrow F(t)
\qquad\text{as }p\to2^+
\]
uniformly for $t$ in every compact subinterval of $(0,\infty)$. Indeed, on the region $t|u|\leq1$ we use
\[
\left|\frac{|tu|^{p-2}-1}{p-2}\right|\leq |\ln(t|u|)|,
\]
whereas on $t|u|>1$ the terms are uniformly controlled by a fixed subcritical power.
Since $F$ is strictly decreasing and $F(1)=0$, we may choose $0<a<1<b$ such that
\[
F(a)>0>F(b).
\]
For $p>2$ sufficiently close to $2$, the same inequalities hold for $F_p$, and hence
\[
a<t_p<b.
\]
Every convergent subsequence of $\{t_p\}$ has a limit $t_*$ satisfying $F(t_*)=0$. By uniqueness of the zero of $F$, $t_*=1$. Consequently,
\[
t_p\longrightarrow1\qquad\text{as }p\to2^+.
\]

By the definitions of $I_p$, $I$ and the inequality
\[
\frac{s^{p-2}-1}{p-2}\geq\ln s,
\qquad s>0,
\]
(see \cite{fengwx2024}), and since $m$ is odd, we have
\begin{eqnarray*}
	I_p(t_pu)-I(t_pu)
	&=&\frac{\beta}{2p}\int_{\mathbb R^N}|t_pu|^p\,dx
	-\frac{\beta}{4}\int_{\mathbb R^N}|t_pu|^2\,dx\\
	&&-\frac{\beta}{2}\int_{\mathbb R^N}
	\left(\frac{|t_pu|^{p-2}-1}{p-2}-\ln|t_pu|\right)|t_pu|^2\,dx\\
	&&-\gamma\int_{\mathbb R^N}\int_0^{t_pu}
	\left[
	\left(\frac{|t|^{p-2}-1}{p-2}\right)^m-(\ln|t|)^m
	\right]t\,dt\,dx\\
	&\leq&\frac{\beta}{2p}\int_{\mathbb R^N}|t_pu|^p\,dx
	-\frac{\beta}{4}\int_{\mathbb R^N}|t_pu|^2\,dx.
\end{eqnarray*}
Since $t_p\to1$ and $u\in H^1(\mathbb R^N)\subset L^2(\mathbb R^N)\cap L^q(\mathbb R^N)$ for some fixed $q>2$, the dominated convergence theorem yields
\[
\int_{\mathbb R^N}|t_pu|^p\,dx\longrightarrow\int_{\mathbb R^N}|u|^2\,dx.
\]
Hence
\[
\limsup_{p\to2^+}\bigl(I_p(t_pu)-I(t_pu)\bigr)\leq0.
\]
Moreover, $u\in\mathcal D$ and $t_p\to1$ imply
\[
I(t_pu)\longrightarrow I(u).
\]
Therefore,
\[
\limsup_{p\to2^+}I_p(t_pu)\leq I(u).
\]
Since $t_pu\in\mathcal N_p$,
\[
c_p\leq I_p(t_pu),
\]
and consequently
\[
\limsup_{p\to2^+}c_p\leq I(u).
\]
Taking the infimum over $u\in\mathcal N$, we obtain
\[
\limsup_{p\to2^+}c_p\leq c.
\]

\medskip
\noindent
{\it Step 3. We prove that $\liminf\limits_{p\to2^+}c_p>0$.}

Let $p_0>2$ and $q_0>2$ be chosen as in \Cref{l1: sec2-5}. Repeating the estimate in \eqref{l1: eq2-6}, but without imposing any upper bound on $I_p(u)$, there exists a constant $C_0>0$, independent of $p\in(2,p_0)$, such that
\begin{equation}\label{eq:uniform-positive-geometry}
	I_p(u)
	\geq
	\frac12\min\{\alpha,V_0\}\|u\|_{H^1(\mathbb R^N)}^2
	-C_0\|u\|_{H^1(\mathbb R^N)}^{q_0},
	\qquad u\in H^1(\mathbb R^N).
\end{equation}
Indeed, on $|u|\leq1$ the two nonlinear primitive terms have the favorable sign, while on $|u|>1$ they are bounded, uniformly for $p\in(2,p_0)$, by a constant multiple of $|u|^{q_0}$; the Sobolev embedding then gives \eqref{eq:uniform-positive-geometry}.

Set
\[
a_0:=\min\{\alpha,V_0\}>0.
\]
Choose $\rho>0$ sufficiently small so that
\[
C_0\rho^{q_0-2}\leq\frac{a_0}{4}.
\]
Then, for every $p\in(2,p_0)$ and every $u\in H^1(\mathbb R^N)$ satisfying $\|u\|_{H^1}=\rho$, \eqref{eq:uniform-positive-geometry} yields
\[
I_p(u)\geq\frac{a_0}{4}\rho^2=:\eta>0.
\]
Recall from the Nehari characterization established above that
\[
c_p
=
\inf_{u\in H^1(\mathbb R^N)\setminus\{0\}}
\max_{t>0}I_p(tu).
\]
For any $u\neq0$, taking
\[
t=\frac{\rho}{\|u\|_{H^1}}
\]
gives $\|tu\|_{H^1}=\rho$, and therefore
\[
\max_{t>0}I_p(tu)\geq\eta.
\]
Taking the infimum over all $u\neq0$, we obtain
\[
c_p\geq\eta>0,
\qquad 2<p<p_0.
\]
Consequently,
\[
\liminf_{p\to2^+}c_p\geq\eta>0.
\]
This completes the proof.\qed

\subsection{Profile decomposition of Palais-Smale sequences}
\noindent
We recall the compactness properties of Palais-Smale sequences associated with the autonomous problem. Due to the lack of compactness of the embedding on $\mathbb R^N$, a bounded Palais-Smale sequence may lose compactness through translations. The profile decomposition below describes this phenomenon by decomposing the sequence into translated solutions of the autonomous equation.
\begin{lemma}\label{l1: sec2-7}
	Let $p_0\in(2,\bar{p})$ be chosen as in \Cref{l1: sec2-5}, and let $\{p_n\}\subset(2,p_0)$ and $\{u_n\}\subset H^1(\R^N)$ satisfy
	\begin{eqnarray*}
		&p_n\to 2^+,\;\;u_n\in\mathcal{N}_{p_n},\nonumber\\
		&\|I'_{p_n}|_{\mathcal N_{p_n}}(u_n)\|_*\to 0,\;\;I_{p_n}(u_n)\to d\in\R\;\;\text{as}\;\;n\to\infty.
	\end{eqnarray*}
	Then up to a subsequence of $\{u_n\}\in\mathcal N_{p_n}$, still denoted by $\{u_n\}$, there exist a weak solution $\widetilde u\in\mathcal D$ of (\ref{l1: eq1-1}), an integer $k\ge0$, nontrivial weak solutions $X^1,\ldots,X^k\in\mathcal D$ of (\ref{l1: eq2-3}), and translations $y_n^i\in\mathbb R^N$ such that
	\begin{eqnarray*}
		&|y_n^i|\to\infty\;\;\text{and}\;\; |y_n^i-y_n^j|\to+\infty\;\;\text{if}\;\;i\ne j,\\
		&u_n-\sum\limits_{i=1}^kX^i(x-y_n^i)\to\widetilde u\;\;\text{in}\;\;H^1(\mathbb{R}^N),\\
		&d=I(\widetilde u)+\sum\limits_{i=1}^kI_\infty(X^i).
	\end{eqnarray*}
	Note, when $k=0$, the above results hold without $X^i$.
\end{lemma}

\noindent
{\bf Proof.}
In what follows, we divide the proof into the following six steps.

\smallskip
\noindent
\emph{Step 1. Uniform boundedness and local Palais‑Smale condition.}

Since $u_n\in\mathcal N_{p_n}$, we have $\varphi_{p_n}(u_n)=\langle I_{p_n}'(u_n),u_n\rangle=0$. Moreover, since $I_{p_n}(u_n)\to d$, there exists $d_1>0$ such that $I_{p_n}(u_n)\le d_1$ for all large enough $n$. Indeed, we can obtain a similar result with (\ref{l1: eq2-5}), that is, 
$$
2d_1
\ge2I_{p_n}(u_n)-\varphi_{p_n}(u_n)
=\frac{\beta}{p_n}\int_{\mathbb R^N}|u_n|^{p_n}dx
+m\gamma\int_{\mathbb{R}^N}\int_0^{u_n}\left(\frac{|t|^{p_n-2}-1}{p_n-2}\right)^{m-1}|t|^{p_n-2}tdtdx.
$$
Using the method for deriving conclusions in \Cref{l1: sec2-5}, we can obtain
\begin{equation}
	\label{l1: eq2-8}
	\sup_n\left(\|u_n\|_{H^1}+\int_{\mathbb R^N}|u_n|^2\left|\frac{|u_n|^{p_n-2}-1}{p_n-2}\right|^m dx
	+\int_{\mathbb R^N}\left|\frac{|u_n|^{p_n}-|u_n|^2}{p_n-2}\right|dx\right)\le 3C(d_1)<\infty.
\end{equation}
Moreover, since $u_n\in\mathcal N_{p_n}$, by the definition of $c_{p_n}$ and \Cref{l1: sec2-6}, one has
$$
d=\lim_{n\to\infty}I_{p_n}(u_n)\ge\liminf_{n\to\infty}c_{p_n}\ge\liminf_{p\to2^+}c_{p}>0.
$$

We first establish a pointwise estimate for the primitive term $\int_0^s\left(\frac{|t|^{p-2}-1}{p-2}\right)^{m-1}|t|^{p-2}tdt$ appearing in \Cref{l1: sec2-5}, which will be used to compare $2I_{p_n}(u_n)$ with the transversality quantity $-\langle\varphi_{p_n}'(u_n),u_n\rangle$. Indeed, according to \Cref{l1: sec2-5}, we know that the primitive term is even and nonnegative. For $0<s\le1$, making the change of variables $t=sr$, using $C_p$ inequality, (\ref{l1: eq2-1}) and some simple transformations, we obtain
\begin{eqnarray*}
	\int_0^s\left(\frac{|t|^{p-2}-1}{p-2}\right)^{m-1}|t|^{p-2}tdt
	& =  &\int_0^1\left|\frac{|sr|^{p-2}-1}{p-2}\right|^{m-1}|sr|^{p-1}srd(sr)\\
	& =  &|s|^p\int_0^1\left|\frac{|s|^{p-2}-1}{p-2}+|s|^{p-2}\frac{|r|^{p-2}-1}{p-2}\right|^{m-1}|r|^{p-2}rdr\\
	& \le&C_{m-1}|s|^p\int_0^1\left[\left(\left|\frac{|s|^{p-2}-1}{p-2}\right|^{m-1}
	+|s|^{(p-2)(m-1)}\left|\frac{|r|^{p-2}-1}{p-2}\right|^{m-1}\right)|r|^{p-2}rdr\right]\\
	& \le&C_{m-1}|s|^p\left(\left|\frac{|s|^{p-2}-1}{p-2}\right|^{m-1}\int_0^1|r|^{p-2}rdr
	+\int_0^1\left|\frac{|r|^{p-2}-1}{p-2}\right|^{m-1}|r|^{p-2}rdr\right)\\
	& \le&C_{m-1}|s|^p\left(\left|\frac{s^{p-2}-1}{p-2}\right|^{m-1}
	+\int_0^1\left|\frac{|r|^{p-2}-1}{p-2}\right|^{m-1}|r|^{p-2}rdr\right)\\
	& \le&C_{m-1}|s|^p\left(\left|\frac{|s|^{p-2}-1}{p-2}\right|^{m-1}
	+\int_0^1|\ln r|^{m-1}rdr\right)\\
	& =  &C_{m-1}|s|^p\left(\left|\frac{|s|^{p-2}-1}{p-2}\right|^{m-1}
	+\frac{(m-1)!}{2^m}\right).
\end{eqnarray*}
For $s>1$, the function $\frac{|t|^{p-2}-1}{p-2}$ is nonnegative and increasing on $(1,s]$. Based on the estimate established above for $0<s\leq1$, we have
\begin{eqnarray*}
	\int_0^s\left(\frac{|t|^{p-2}-1}{p-2}\right)^{m-1}|t|^{p-2}tdt
	& =  &\int_0^1\left(\frac{|t|^{p-2}-1}{p-2}\right)^{m-1}|t|^{p-2}tdt
	+\int_1^s\left(\frac{|t|^{p-2}-1}{p-2}\right)^{m-1}|t|^{p-2}tdt\\
	& \le&C_{m-1}\frac{(m-1)!}{2^m}+\left(\frac{|s|^{p-2}-1}{p-2}\right)^{m-1}\int_1^s|t|^{p-2}tdt\\
	& \le&C_{m-1}\frac{(m-1)!}{2^m}+\left(\frac{|s|^{p-2}-1}{p-2}\right)^{m-1}|s|^p.
\end{eqnarray*}
Therefore, for every $2<p<p_0$ and every $s>0$, there exists constant $C_4$, depend on $m$, such that
\begin{equation}
	\label{l1: eq2-9}
	0\le\int_0^s\left(\frac{|t|^{p-2}-1}{p-2}\right)^{m-1}|t|^{p-2}tdt
	\le C_4C_{m-1}|s|^p\left(1+\left|\frac{|s|^{p-2}-1}{p-2}	\right|^{m-1}\right).
\end{equation}

By the definition of the restricted derivative, there exist
$\lambda_n\in\mathbb R$ such that
\begin{equation}
	\label{l1: eq2-10}
	R_n:=I_{p_n}'(u_n)-\lambda_n\varphi_{p_n}'(u_n)\longrightarrow0\qquad\hbox{in }H^{-1}(\mathbb R^N).
\end{equation}
Using the definition of Gateaux derivative and $\varphi_p$, there exists
\begin{eqnarray*}
	\langle\varphi_{p_n}'(u_n),u_n\rangle
	&=& 2\int_{\R^N}\alpha |\nabla u_n|^2dx+2\int_{\R^N}V(x)u_n^2dx-2\beta\int_{\R^N}\frac{|u_n|^{p_n}-u_n^2}{p_n-2}dx
	-\beta\int_{\R^N}|u_n|^{p_n}dx\\
	& &-2\gamma\int_{\R^N} \left(\frac{|u_n|^{p_n-2}-1}{p_n-2}\right)^m u_n^2dx
	-m\gamma\int_{\R^N} \left(\frac{|u_n|^{p_n-2}-1}{p_n-2}\right)^{m-1} |u_n|^pdx\\
	&=&2\varphi_{p_n}-\beta\int_{\R^N}|u_n|^{p_n}dx-m\gamma\int_{\R^N} \left(\frac{|u_n|^{p_n-2}-1}{p_n-2}\right)^{m-1} |u_n|^pdx
\end{eqnarray*}
Let
\begin{equation*}
	D_n:=\int_{\mathbb R^N}\left[\beta+m\gamma\left(\frac{|u_n|^{p_n-2}-1}{p_n-2}\right)^{m-1}\right]|u_n|^{p_n}dx.
\end{equation*}

By \eqref{l1: eq2-5} and \eqref{l1: eq2-9}, for every $u\in\mathcal N_p$, there exists constant $C_5$ such that 
$$
2I_{p_n}(u_n)-\varphi_{p_n}(u_n)
\leq C_5\int_{\mathbb R^N}\left[\beta+m\gamma\left|\frac{|u_n|^{p_n-2}-1}{p_n-2}\right|^{m-1}\right]|u_n|^{p_n}dx=C_5D_n.
$$
Consequently, according to assumptions, we have
\[
D_n\ge\frac{2I_{p_n}(u_n)}{C_5}\longrightarrow \frac{2d}{C_5}>0.
\]
Applying $u_n$ to both sides of the above equation \eqref{l1: eq2-10} and using $\langle I_{p_n}'(u_n),u_n\rangle=0$, we obtain
\begin{equation*}
	\lambda_n\longrightarrow0.
\end{equation*}
Let $\psi\in C_c^\infty(\mathbb R^N)$ and set $K:=\operatorname{supp}\psi$.
For $y\in\mathbb R^N$, define $\psi_y(x):=\psi(x-y)$. By
\eqref{l1: eq2-10},
\[
\langle I_{p_n}'(u_n),\psi_y\rangle
=\langle R_n,\psi_y\rangle+\lambda_n\langle\varphi_{p_n}'(u_n),\psi_y\rangle.
\]
Since translations preserve the $H^1$ norm, it is enough to prove that
\[
\sup_{y\in\mathbb R^N}
\left|\langle\varphi_{p_n}'(u_n),\psi_y\rangle\right|
\]
is uniformly bounded with respect to $n$.

For $2<p<p_0$, a direct differentiation gives
\begin{eqnarray*}
	\langle\varphi_{p_n}'(u_n),\psi_y\rangle
	&=&2\int_{\mathbb R^N}\left(\alpha\nabla u_n\cdot\nabla\psi_y+V(x)u_n\psi_y\right)dx
	-2\beta\int_{\mathbb R^N}\frac{|u_n|^{p_n-2}-1}{p_n-2}u_n\psi_y\,dx\\
	& &-2\gamma\int_{\mathbb R^N}\left(\frac{|u_n|^{p_n-2}-1}{p_n-2}\right)^mu_n\psi_y\,dx
	-\int_{\mathbb R^N}\left(\beta+m\gamma\left(\frac{|u_n|^{p_n-2}-1}{p_n-2}\right)^{m-1}\right)|u_n|^{p_n-2}u_n\psi_y\,dx.
\end{eqnarray*}

Taking absolute values and absorbing the fixed coefficients into the
constants, the nonlinear terms are controlled by
\[
s\left(1+\left|\frac{s^{p-2}-1}{p-2}\right|^m\right)
+s^{p-1}\left(1+\left|\frac{s^{p-2}-1}{p-2}\right|^{m-1}\right),\qquad s>0.
\]
For $0<s\le1$,
\[
\left|\frac{s^{p-2}-1}{p-2}\right|\le |\ln s|,\qquad s^{p-1}\le s,
\]
and hence, for every fixed $\varepsilon\in(0,1)$,
\[
s\left(1+\left|\frac{s^{p-2}-1}{p-2}\right|^m\right)
+s^{p-1}\left(1+\left|\frac{s^{p-2}-1}{p-2}\right|^{m-1}\right)
\le C_6s^\varepsilon.
\]
For $s>1$, by \Cref{l1: sec2-3} and $m(p-2)\le q-2$,
\[
\left|\frac{s^{p-2}-1}{p-2}\right|^{m-1}
\le C_{\bar p,m}^{\frac{m-1}{m}}s^{\frac{(q-2)(m-1)}{m}},
\]
and, since
\[
p-1+\frac{(q-2)(m-1)}m\le q-1,
\]
we obtain
\[
s\left(1+\left|\frac{s^{p-2}-1}{p-2}\right|^m\right)
+s^{p-1}\left(1+\left|\frac{s^{p-2}-1}{p-2}\right|^{m-1}\right)
\le C_7s^{q-1}.
\]
Therefore, using the H\"older inequality, the global boundedness of
$\{u_n\}$ in $H^1(\mathbb R^N)$, and the translation invariance of the
norms of $\psi_y$, we have
\begin{align*}
	\left|\langle\varphi_{p_n}'(u_n),\psi_y\rangle\right|
	\le{}&2\alpha\|\nabla u_n\|_{L^2}\|\nabla\psi\|_{L^2}
	+2V^*\|u_n\|_{L^2}\|\psi\|_{L^2}\\
	&+C_6\int_{K+y\cap\{|u_n|\le1\}}|\psi_y|\,dx
	+C_7\int_{K+y\cap\{|u_n|>1\}}|u_n|^{q-1}|\psi_y|\,dx\\
	\le{}&C\|u_n\|_{H^1}\|\psi\|_{H^1}
	+C_6|K|\|\psi\|_{L^\infty}
	+C_7\|u_n\|_{L^q}^{q-1}\|\psi\|_{L^q}
	\le C_8(\psi),
\end{align*}
where $C_8(\psi)>0$ is independent of $n$ and $y$. Consequently,
\begin{align*}
	\sup_{y\in\mathbb R^N}
	\left|\langle I_{p_n}'(u_n),\psi_y\rangle\right|
	\le \|R_n\|_{H^{-1}(\mathbb R^N)}\|\psi\|_{H^1(\mathbb R^N)}
	+|\lambda_n|C_8(\psi)\longrightarrow0.
\end{align*}
Thus,
\begin{equation}
	\label{l1: eq2-13}
	\sup_{y\in\mathbb R^N}
	\left|\langle I_{p_n}'(u_n),\psi(\cdot-y)\rangle\right|
	\longrightarrow0
	\qquad\text{for every }\psi\in C_c^\infty(\mathbb R^N).
\end{equation}
In particular, taking $y=0$ gives the local Palais--Smale condition for every fixed test function.

\smallskip
\noindent
\emph{Step 2. Profile extraction and identification of the profiles.}

Since $\{u_n\}$ is bounded in $H^1(\mathbb R^N)$, there exist a
subsequence of $\{u_n\}$, still denoted by $\{u_n\}$, and a function
$\widetilde u\in H^1(\mathbb R^N)$ such that
\begin{eqnarray*}
	\begin{cases}
		u_n\rightharpoonup\widetilde u
		&\quad\text{in }H^1(\mathbb R^N),\\
		u_n(x)\to\widetilde u(x)
		&\quad\text{for a.e. }x\in\mathbb R^N,\\
		u_n\to\widetilde u
		&\quad\text{in }L^s_{\mathrm{loc}}(\mathbb R^N),
		\;\;s\in(2,2^*).
	\end{cases}
\end{eqnarray*}

For a.e. $x\in\mathbb R^N$ with $\widetilde u(x)\neq0$, since
$p_n\to2^+$ and $u_n(x)\to\widetilde u(x)$, one has
\[
|u_n|^2\left|\frac{|u_n|^{p_n-2}-1}{p_n-2}\right|^m\longrightarrow|\widetilde u|^2|\ln|\widetilde u||^m.
\]
If $\widetilde u(x)=0$, its total value is zero. Hence, by Fatou lemma and \eqref{l1: eq2-8},
\begin{equation}
	\label{l1: eq2-14}
	\int_{\mathbb R^N}|\widetilde u|^2|\ln|\widetilde u||^m dx
	\leq\liminf_{n\to\infty}\int_{\mathbb R^N}|u_n|^2\left|\frac{|u_n|^{p_n-2}-1}{p_n-2}\right|^m dx<\infty.
\end{equation}
Therefore, $\widetilde u\in\mathcal D$.

Next, we show that $\widetilde u$ is a weak solution of (\ref{l1: eq1-1}). By the estimates established in Step 1, for every compact set $K\subset\mathbb R^N$, the sequence  $\left[\beta\frac{|u_n|^{p_n-2}-1}{p_n-2}+\gamma\left(\frac{|u_n|^{p_n-2}-1}{p_n-2}\right)^m\right]u_n$ is  uniformly integrable on $K$. Moreover, based on the analysis of $\widetilde u\in\mathcal D$, we can see that the nonlinear terms converges almost everywhere to the logarithmic terms in \eqref{l1: eq2-3}. Therefore, by the Vitali theorem \cite{Brezis2011},
%定义在可测集E上的可积函数列{fn(x)}满足两种收敛特性：依测度收敛于极限函数f(x)，同时其积分具有等度绝对连续性,极限函数f(x)的可积性得到保证，且原函数列的积分值将收敛于该极限函数的积分%
\begin{equation}
	\label{l1: eq2-15}
	\left[\beta\frac{|u_n|^{p_n-2}-1}{p_n-2}+\gamma\left(\frac{|u_n|^{p_n-2}-1}{p_n-2}\right)^m\right]u_n
	\longrightarrow
	\left[\beta\ln|\widetilde u|+\gamma(\ln|\widetilde u|)^m\right]\widetilde u
\end{equation}
in $L^1_{\mathrm{loc}}(\mathbb R^N)$.

We now pass to the limit in the weak equation. It follows from the definition of $\varphi_{p_n}$ and \eqref{l1: eq2-13} that
\begin{align*}
	\alpha\int_{\mathbb R^N}\nabla u_n\cdot\nabla\psi\,dx+\int_{\mathbb R^N}V(x)u_n\psi\,dx
	-\int_{\mathbb R^N}\left[\beta\frac{|u_n|^{p_n-2}-1}{p_n-2}
	+\gamma\left(\frac{|u_n|^{p_n-2}-1}{p_n-2}\right)^m\right]u_n\psi\,dx\longrightarrow0.
\end{align*}

Since
\[
u_n\rightharpoonup\widetilde u\qquad\text{in }H^1(\mathbb R^N),
\]
we have
\[
\int_{\mathbb R^N}\nabla u_n\cdot\nabla\psi\,dx\longrightarrow\int_{\mathbb R^N}\nabla\widetilde u\cdot\nabla\psi\,dx.
\]
Moreover, since $u_n\to\widetilde u$ in $L^2_{\mathrm{loc}}(\mathbb R^N)$ and $V\in L^\infty(\mathbb R^N)$,
\[
\int_{\mathbb R^N}V(x)u_n\psi\,dx\longrightarrow\int_{\mathbb R^N}V(x)\widetilde u\psi\,dx.
\]
Finally, by \eqref{l1: eq2-15} and the boundedness of $\psi$,
\begin{align*}
	\int_{\mathbb R^N}\left[\beta\frac{|u_n|^{p_n-2}-1}{p_n-2}
	+\gamma\left(\frac{|u_n|^{p_n-2}-1}{p_n-2}\right)^m	\right]u_n\psi\,dx\longrightarrow
	\int_{\mathbb R^N}\left[\beta\ln|\widetilde u|+\gamma(\ln|\widetilde u|)^m\right]\widetilde u\psi\,dx.
\end{align*}

Hence,
\[
\alpha\int_{\mathbb R^N}
\nabla\widetilde u\cdot\nabla\psi\,dx
+\int_{\mathbb R^N}
V(x)\widetilde u\psi\,dx
-
\int_{\mathbb R^N}
\left[
\beta\ln|\widetilde u|
+\gamma(\ln|\widetilde u|)^m
\right]\widetilde u\psi\,dx
=0
\]
for every $\psi\in C_c^\infty(\mathbb R^N)$. Therefore,
$\widetilde u$ is a weak solution of (\ref{l1: eq1-1}).

Before extracting the profiles at infinity, we first record that
every weak solution $U\in\mathcal D$ of either (\ref{l1: eq1-1}) or
(\ref{l1: eq2-3}) can be tested by itself. Indeed, let $T_M$ be the
standard truncation. Choose $\chi\in C_c^\infty(\mathbb R^N)$ such
that
\[
0\leq\chi\leq1,\qquad \chi=1\ \text{in }B_1(0),\qquad \chi=0\ \text{outside }B_2(0),
\]
and define
\[
\chi_R(x):=\chi\left(\frac{x}{R}\right).
\]
Then
\[
0\leq\chi_R\leq1,\qquad \chi_R=1\ \text{in }B_R(0),\qquad \chi_R=0\ \text{outside }B_{2R}(0),
\]
and
\[
|\nabla\chi_R(x)|=\frac{1}{R}|\nabla \chi\left(\frac{x}{R}\right)|\leq\frac{\max|\nabla\chi|}{R}.
\]
The function $\chi_R^2T_M(U)$ is an admissible
$H_0^1(B_{2R})$ test function by the standard approximation
argument. Hence,
\begin{align*}
	&\alpha\int_{\mathbb R^N}\chi_R^2T_M'(U)|\nabla U|^2dx
	+2\alpha\int_{\mathbb R^N}\chi_RT_M(U)\nabla U\cdot\nabla\chi_R\,dx+\int_{\mathbb R^N}	W(x)U\chi_R^2T_M(U)\,dx\\
	&=\int_{\mathbb R^N}\left[\beta\ln|U|+\gamma(\ln|U|)^m\right]U\chi_R^2T_M(U)\,dx.
\end{align*}
Since $U\in\mathcal D$,
\[
U^2\left(1+|\ln|U||^m\right)\in L^1(\mathbb R^N).
\]
Thus, letting $M\to\infty$, we obtain
\begin{align*}
	\alpha\int_{\mathbb R^N}\chi_R^2|\nabla U|^2dx+\int_{\mathbb R^N}W(x)\chi_R^2U^2dx
	+2\alpha\int_{\mathbb R^N}\chi_RU\nabla U\cdot\nabla\chi_R\,dx
	=\int_{\mathbb R^N}\left[\beta\ln|U|+\gamma(\ln|U|)^m\right]U^2\chi_R^2\,dx.
\end{align*}
Moreover, using the definition of $\chi_R$ and Cauchy-Schwarz inequality, we have
\begin{eqnarray*}
	\left|\int_{\mathbb R^N}\chi_RU\nabla U\cdot\nabla\chi_R\,dx\right|
	&\leq& \int_{\mathbb R^N}\left|\chi_RU\nabla U\cdot\nabla\chi_R\right|\,dx\\
	&\leq& \frac{\max|\nabla\chi|}{R}\int_{\mathbb R^N}\left|U\nabla U\right|\,dx
	=\frac{\max|\nabla\chi|}{R}\|U\|_{L^2}\|\nabla U\|_{L^2}\longrightarrow0,\;\;\text{as}\;\;R\to \infty.
\end{eqnarray*}
Therefore, letting $R\to\infty$, we obtain
\[
\int_{\mathbb R^N}\left(\alpha|\nabla U|^2+W(x)U^2\right)dx
=\int_{\mathbb R^N}\left[\beta\ln|U|+\gamma(\ln|U|)^m\right]U^2\,dx,
\]
where $W(x)=V(x)$ for (\ref{l1: eq1-1}) and $W(x)=1$ for (\ref{l1: eq2-3}).
In particular, since $\widetilde u$ is a weak solution of
(\ref{l1: eq1-1}),
\begin{equation*}
	\int_{\mathbb R^N}\left(\alpha|\nabla\widetilde u|^2+V(x)\widetilde u^2\right)dx
	=\int_{\mathbb R^N}\left[\beta\ln|\widetilde u|+\gamma(\ln|\widetilde u|)^m\right]\widetilde u^2\,dx.
\end{equation*}

Set 
$$
z_n^0:=u_n-\widetilde u. 
$$
We distinguish the following two cases.

\medskip
\noindent
\emph{Case 1. Vanishing occurs.}
Suppose that
\begin{equation}
	\label{l1: eq2-17}
	\sup_{y\in\mathbb R^N}\int_{B_1(y)}|z_n^0|^2\,dx\longrightarrow0.
\end{equation}
Since $\{z_n^0\}$ is bounded in $H^1(\mathbb R^N)$, applying the Lemma I.1 in Lions' concentration--compactness principle \cite{Lions1984-2} with $p=q=2$ and $R=1$, yields
\[
z_n^0\longrightarrow0
\qquad\text{in }L^s(\mathbb R^N),
\qquad 2<s<2^*.
\]
In this case, no profile at infinity is required.

\medskip
\noindent
\emph{Case 2. Vanishing does not occur.}
Suppose that \eqref{l1: eq2-17} does not hold. Then, up to a subsequence, there exists $\eta>0$ such that $\sup\limits_{y\in\mathbb R^N}\int_{B_1(y)}|z_n^0|^2\,dx\geq \eta$ for every $n$, which implies that we may choose $y_n^1\in\mathbb R^N$ such that $\int_{B_1(y_n^1)}|z_n^0|^2\,dx\geq\eta$. 

Assume $\{y_n^1\}$ were bounded, then there would exist $R_0>0$ such that $B_1(y_n^1)\subset B_{R_0}$ for all $n$. According to the boundedness of $\{z_n^0\}\in H^1(\mathbb R^N)$ and the local compactness of embedding, one has $z_n^0\longrightarrow0\;\;\text{in }\;\;L^2(B_{R_0})$, and hence
\[
\int_{B_1(y_n^1)}|z_n^0|^2\,dx\longrightarrow0,
\]
which contradicts with the assumption. Therefore,
\[
|y_n^1|\longrightarrow\infty.
\]

Since $\{z_n^0(x+y_n^1)\}$ is bounded in
$H^1(\mathbb R^N)$, there exist a subsequence of $\{z_n^0(x+y_n^1)\}$, still denoted by $\{z_n^0(x+y_n^1)\}$, and a function $X^1\in H^1(\mathbb R^N)$ such that
\[
z_n^0(x+y_n^1)\rightharpoonup X^1\quad\text{in }H^1(\mathbb R^N)\qquad\text{and}\qquad
z_n^0(x+y_n^1)\longrightarrow X^1\quad\text{in }L^2_{\mathrm{loc}}(\mathbb R^N).
\]
Consequently, let $\xi=x+y_n^1$, then by the relationship of $y_n^1$ and $\eta$, we have
\[
\int_{B_1}|X^1(x)|^2dx=\lim_{n\to\infty}\int_{B_1}|z_n^0(x+y_n^1)|^2dx
=\lim_{n\to\infty}\int_{B_1(y_n^1)}|z_n^0(\xi)|^2d\xi\geq\eta,
\]
and hence $X^1\neq0$.

For every $\psi\in C_c^\infty(\mathbb R^N)$, taking
$\psi(x-y_n^1)$ in \eqref{l1: eq2-13} and making the change of
variables $\xi=x-y_n^1$, we obtain
\[
\begin{aligned}
	&\alpha\int_{\mathbb R^N}
	\nabla u_n(\xi+y_n^1)\cdot\nabla\psi(\xi)\,d\xi
	+\int_{\mathbb R^N}
	V(\xi+y_n^1)u_n(\xi+y_n^1)\psi(\xi)\,d\xi\\
	&\quad
	-\int_{\mathbb R^N}
	\left[
	\beta\frac{|u_n(\xi+y_n^1)|^{p_n-2}-1}{p_n-2}
	+\gamma
	\left(
	\frac{|u_n(\xi+y_n^1)|^{p_n-2}-1}{p_n-2}
	\right)^m
	\right]
	u_n(\xi+y_n^1)\psi(\xi)\,d\xi
	\longrightarrow0.
\end{aligned}
\]
Since $z_n^0:=u_n-\widetilde u$ and $\tilde{u}\in H^1(\mathbb{R}^N)$, there exists
\[
u_n(x+y_n^1)\rightharpoonup X^1
\quad\text{in }H^1(\mathbb R^N),
\]
\[
u_n(x+y_n^1)\to X^1
\quad\text{in }L^s_{\mathrm{loc}}(\mathbb R^N),
\qquad 2<s<2^*,
\]
and, by $(AC_1)$,
\[
V(x+y_n^1)\to1
\quad\text{uniformly on compact subsets of }\mathbb R^N.
\]
Consequently, according to the Vitali theorem \cite{Brezis2011} and above results, we have the following limit equation
\[
\alpha\int_{\mathbb R^N}\nabla X^1\cdot\nabla\psi\,dx+\int_{\mathbb R^N}X^1\psi\,dx
-\int_{\mathbb R^N}\left[\beta\ln|X^1|+\gamma(\ln|X^1|)^m\right]X^1\psi\,dx=0,\;\;\forall \psi\in C_c^\infty(\mathbb{R}^N).
\]
Hence, $X^1$ is a weak solution of the autonomous problem (\ref{l1: eq2-3}). Moreover,  by the translation invariance of the integrals in \eqref{l1: eq2-8} and Fatou's lemma, the same argument as in \eqref{l1: eq2-14} gives $\int_{\mathbb R^N}|X^1|^2|\ln|X^1||^m dx<\infty$. Then, we can derive $X^1\in\mathcal D$. 

Set
\[
z_n^1:=z_n^0-X^1(x-y_n^1).
\]
Repeating the above extraction procedure, we obtain nontrivial profiles $X^1,\ldots,X^j\in\mathcal D$, each of which is a weak solution of the autonomous problem, together with translations $y_n^1,\ldots,y_n^j$ satisfying
\[
|y_n^i|\to\infty,\qquad|y_n^i-y_n^\ell|\to\infty\qquad(i\neq\ell),
\]
and the usual Hilbert-space orthogonality gives
\begin{equation}
	\label{l1: eq2-18}
	\|u_n\|_{{H^1}(\mathbb{R}^N)}^2
	=\|\widetilde u\|_{{H^1}(\mathbb{R}^N)}^2+\sum_{i=1}^{j}\|X^i\|_{{H^1}(\mathbb{R}^N)}^2
	+\|z_n^j\|_{{H^1}(\mathbb{R}^N)}^2+o(1).
\end{equation}
%每抽出一个 profile，就会从总的 \(H^1\) 质量中真实地拿走一部分质量。%
We next prove that the above extraction terminates after finitely many steps. To show the results, it suffices to prove that every nontrivial autonomous profile has a uniform positive lower bound in $H^1(\mathbb R^N)$. Let $X\in\mathcal D$ be any nontrivial weak solution of (\ref{l1: eq2-3}). Applying the self-testing identity established above to $X$, we obtain
\begin{equation*}
	\int_{\mathbb R^N}\left(\alpha|\nabla X|^2+|X|^2\right)dx
	=\int_{\mathbb R^N}\left[\beta\ln|X|+\gamma(\ln|X|)^m\right]|X|^2dx.
\end{equation*}
Since $m$ is odd,
\[
\beta\ln s+\gamma(\ln s)^m\leq0,\qquad 0<s\leq1.
\]
Moreover, for the fixed $q\in(2,2^*)$ chosen above, there exists a constant $C_q>0$ such that
\[
\beta\ln s+\gamma(\ln s)^m\leq C_qs^{q-2},\qquad s>1.
\]
Hence, by the Sobolev embedding inequality,
\begin{eqnarray*}
	\min\{\alpha, 1\}\|X\|_{{H^1}(\mathbb{R}^N)}^2
	&\leq& \int_{\mathbb R^N}\left(\alpha|\nabla X|^2+|X|^2\right)dx\\
	& =  & \int_{\mathbb R^N}\left[\beta\ln|X|+\gamma(\ln|X|)^m\right]|X|^2dx\\
	&\leq& \int_{|X|>1}\left[\beta\ln|X|+\gamma(\ln|X|)^m\right]|X|^2dx\\
	&\leq& C_q\int_{\mathbb R^N}|X|^qdx
	\leq C_qC_1\|X\|_{{H^1}(\mathbb{R}^N)}^q.
\end{eqnarray*}
Since $X\neq0$ and $q>2$, we obtain
\[
\|X\|_{{H^1}(\mathbb{R}^N)}\geq\left[\frac{\min\{\alpha, 1\}}{C_qC_1}\right]^{\frac{1}{q-2}}>0.
\]
Therefore, every nontrivial autonomous profile has a uniform positive lower bound in $H^1(\mathbb R^N)$. Combining this with \eqref{l1: eq2-18} and the boundedness of $\{u_n\}$, if $j$ profiles have been extracted, then
$$
0<\sum_{i=1}^j\|X^i\|_{{H^1}(\mathbb{R}^N)}^2\leq \|u_n\|_{{H^1}(\mathbb{R}^N)}^2\leq C(d_1,d_2),
$$
which implies the extraction can produce only finitely many nontrivial profiles.

Hence, for some $k\geq0$, set
\[
v_n
:=
u_n-\widetilde u
-\sum_{i=1}^kX^i(x-y_n^i).
\]
By the maximum of the extraction, the remainder $v_n$ must satisfy
\[
\sup_{y\in\mathbb R^N}\int_{B_1(y)}|v_n|^2\,dx\to0;
\]
otherwise, the preceding extraction procedure would produce one more nontrivial profile, contradicting the maximum of $k$. It follows from the Lemma I.1 in Lions' concentration--compactness principle \cite{Lions1984-2} that
\[
v_n\longrightarrow0
\qquad\text{in }L^s(\mathbb R^N),
\qquad 2<s<2^*.
\]

\smallskip
\noindent
\emph{Step 3. A one-sided high-order logarithmic splitting inequality.}

For convenience, set
\[
H_n(s):=\left[\beta\frac{s^{p_n-2}-1}{p_n-2}+\gamma\left(\frac{s^{p_n-2}-1}{p_n-2}\right)^m\right]s^2,\qquad s\ge0,
\]
and
\[
H_0(s):=\bigl[\beta\ln s+\gamma(\ln s)^m\bigr]s^2,\qquad s>0,\qquad H_0(0):=0.
\]
The logarithmic nonlinearity does not satisfy the standard Brezis-Lieb splitting property directly. Therefore, instead of applying the classical splitting lemma, we establish the splitting property through the uniform integrability estimate and Vitali theorem \cite{Brezis2011}. We first present an estimate similar to the one given earlier, which will be needed in the proof that follows. For $0<s\leq1$, as in Step~1, for every fixed $\varepsilon\in(0,1)$,
\[
|H_n(s)|\leq C_6s^2\left(|\ln s|+|\ln s|^m\right)\leq C_6s^{1+\varepsilon}.
\]
For $s>1$, as in Step~2, 
$$
H_n^+(s)\leq (\beta+\gamma)C_{\bar p,m}s^q.
$$
Let $K\subset\mathbb R^N$ be bounded. Suppose that
\begin{eqnarray*}
	\begin{cases}
		z_n\to z\quad\text{a.e. in }K,\\
		z_n\to z\quad\text{in }L^2(K)\cap L^q(K).
	\end{cases}
\end{eqnarray*}
By the estimate established above, the sequence $\{H_n(|z_n|)\}$ is uniformly integrable on $K$. Moreover, since $H_n(|z_n|)\longrightarrow H_0(|z|)$ a.e. in $K$ when $p_n\to2^+$, by Vitali theorem \cite{Brezis2011}, there exists
\[
H_n(|z_n|)\longrightarrow H_0(|z|)\qquad\text{in }L^1(K).
\]

Fix $R>1$ and set
\[
\Omega_{n,R}
:=
B_R(0)\cup\bigcup_{i=1}^kB_R(y_n^i).
\]
Since
\[
|y_n^i|\to\infty,\qquad|y_n^i-y_n^\ell|\to\infty\quad(i\neq\ell),
\]
the components of $\Omega_{n,R}$ are mutually disjoint for all
sufficiently large $n$. Moreover, by the local convergences obtained
in Step~2,
\[
u_n\to\widetilde u \qquad\text{in }L^2(B_R(0))\cap L^q(B_R(0)),
\]
and, for every $i=1,\ldots,k$,
\[
u_n(x+y_n^i)\to X^i \qquad\text{in }L^2(B_R(0))\cap L^q(B_R(0)).
\]
Therefore, applying the local $L^1$ convergence established above on
each component of $\Omega_{n,R}$, we obtain
\[
\lim_{n\to\infty}\int_{\Omega_{n,R}}H_n(|u_n|)\,dx
=
\int_{B_R(0)}H_0(|\widetilde u|)\,dx+\sum_{i=1}^k\int_{B_R(0)}H_0(|X^i|)\,dx.
\]

On the complement, since $H_n\leq H_n^+$,
\[
\int_{\Omega_{n,R}^c}H_n(|u_n|)\,dx
\leq
\int_{\Omega_{n,R}^c}H_n^+(|u_n|)\,dx
\leq
(\beta+\gamma)C_{\bar p,m}\int_{\Omega_{n,R}^c}|u_n|^q\,dx.
\]
By the decomposition obtained in Step~2,
\[
u_n
=
\widetilde u+\sum_{i=1}^kX^i(x-y_n^i)+v_n,\qquad v_n\to0\quad\text{in }L^q(\mathbb R^N).
\]
Hence, using the generalization of the $C_p$ inequality and some simple transformations, we have
\begin{eqnarray*}
	\int_{\Omega_{n,R}^c}|u_n|^q\,dx
	& =  &\int_{\Omega_{n,R}^c}|\widetilde u+\sum_{i=1}^kX^i(x-y_n^i)+v_n|^q\,dx\\
	&\leq&(k+2)^{q-1}\left(\int_{\Omega_{n,R}^c}|\widetilde u|^qdx+\int_{\Omega_{n,R}^c}\sum_{i=1}^k|X^i(x-y_n^i)|^qdx
	+\int_{\Omega_{n,R}^c}|v_n|^q\,dx\right)\\
	&\leq&
	(k+2)^{q-1}\bigg(\int_{\mathbb R^N\setminus B_R(0)}|\widetilde u|^q\,dx
	+\sum_{i=1}^k\int_{\mathbb R^N\setminus B_R(0)}|X^i|^q\,dx+\int_{\mathbb{R}^N}|v_n|^q\,dx\bigg).
\end{eqnarray*}

Since $\widetilde u,X^1,\ldots,X^k\in H^1(\mathbb R^N) \hookrightarrow L^q(\mathbb R^N)$ and $v_n\to0$ in $L^q(\mathbb R^N)$, it follows that
\[
\lim_{R\to\infty}
\limsup_{n\to\infty}
\int_{\Omega_{n,R}^c}|u_n|^q\,dx
=0.
\]

Combining the estimates on $\Omega_{n,R}$ and
$\Omega_{n,R}^c$, and then letting $R\to\infty$, we obtain
\begin{equation*}
	\limsup_{n\to\infty}
	\int_{\mathbb R^N}H_n(|u_n|)\,dx
	\leq
	\int_{\mathbb R^N}H_0(|\widetilde u|)\,dx
	+
	\sum_{i=1}^k
	\int_{\mathbb R^N}H_0(|X^i|)\,dx.
\end{equation*}
\noindent
\emph{Step 4. Strong convergence of the remainder.}

By the orthogonality of the profiles and the decomposition obtained in Step 2, the gradient and unweighted $L^2$ parts split in the usual way. We also record the corresponding splitting for the potential term. Indeed, since $V\in L^\infty(\mathbb R^N)$ and, by $(AC_1)$,
\[
V(\cdot+y_n^i)\longrightarrow1
\qquad\text{locally uniformly in }\mathbb R^N,
\]
a standard cut-off argument yields
\[
\int_{\mathbb R^N}V(x)|X^i(x-y_n^i)|^2\,dx
\longrightarrow
\int_{\mathbb R^N}|X^i|^2\,dx.
\]
Moreover, all weighted cross terms between distinct profiles, between the profiles and $\widetilde u$, and between these functions and the remainder tend to zero, by the separation
$|y_n^i-y_n^j|\to\infty$, the weak orthogonality in the profile extraction, and the boundedness of $V$. Consequently, we have
\begin{eqnarray*}
	\int_{\mathbb R^N}\left(\alpha|\nabla u_n|^2+V(x)u_n^2\right)\,dx
	&=&	\int_{\mathbb R^N}\left(\alpha|\nabla \tilde u|^2+V(x)\tilde u^2\right)\,dx
	+\sum_{i=1}^{k}\int_{\mathbb R^N}\left(\alpha|\nabla X^i|^2+|X^i|^2\right)\,dx\nonumber\\
	& &	+\int_{\mathbb R^N}	\left(\alpha|\nabla v_n|^2+V(x)v_n^2\right)\,dx+o(1).
\end{eqnarray*}
Since \(u_n\in\mathcal N_{p_n}\), it satisfies
\[
\int_{\mathbb R^N}\left(\alpha|\nabla u_n|^2+V(x)u_n^2\right)\,dx=\int_{\mathbb R^N}H_n(|u_n|)\,dx .
\]

Moreover, the limit profile \(\tilde u\) and each \(X^i\) satisfy the
corresponding limiting equations. Therefore,
$$
\int_{\mathbb R^N}\left(\alpha|\nabla\tilde u|^2+V(x)\tilde u^2\right)\,dx=\int_{\mathbb R^N}H_0(|\tilde u|)\,dx
\quad\text{and}\quad
\int_{\mathbb R^N}\left(\alpha|\nabla X^i|^2+|X^i|^2\right)\,dx=\int_{\mathbb R^N}H_0(|X^i|)\,dx .
$$

Hence, combining the above decomposition with the estimate obtained in
Step 3, we obtain
\begin{eqnarray*}
	\limsup_{n\rightarrow\infty}\int_{\mathbb R^N}\left(\alpha|\nabla v_n|^2+V(x)v_n^2\right)\,dx
	&\leq& \limsup_{n\rightarrow\infty}\Bigg[\int_{\mathbb R^N}\left(\alpha|\nabla u_n|^2+V(x)u_n^2\right)\,dx
	-\int_{\mathbb R^N}\left(\alpha|\nabla \tilde u|^2+V(x)\tilde u^2\right)\,dx\\
	&    & -\sum_{i=1}^{k}\int_{\mathbb R^N}\left(\alpha|\nabla X^i|^2+|X^i|^2\right)\,dx\Bigg]\\
	&\leq& \limsup_{n\rightarrow\infty}\Bigg[\int_{\mathbb R^N}H_n(|u_n|)\,dx
	-\int_{\mathbb R^N}H_0(|\tilde u|)\,dx-\sum_{i=1}^{k}\int_{\mathbb R^N}H_0(|X^i|)\,dx\Bigg]\\
	&\leq&0 .
\end{eqnarray*}

Therefore,
\[
\int_{\mathbb R^N}\left(\alpha|\nabla v_n|^2+V(x)v_n^2\right)\,dx\rightarrow0 .
\]

Consequently,
\[
v_n\rightarrow0 \qquad\text{strongly in }H^1(\mathbb R^N).
\]

\noindent
\emph{Step 5. Splitting of the nonlinear terms.}

From the strong convergence results of $v_n$ in Step 4 and $u_n$ satisfies the Nehari identity $\langle I'_{p_n}(u_n),u_n\rangle=0$, for $n\to \infty$, together with the orthogonality of the profiles obtained in Step 2, this implies the splitting of the quadratic part
\[
\begin{aligned}
	\int_{\mathbb R^N}H_n(|u_n|)\,dx
	&=\int_{\mathbb R^N}\left(\alpha|\nabla u_n|^2+V(x)u_n^2\right)\,dx\\
	&\longrightarrow \int_{\mathbb R^N}\left(\alpha|\nabla\widetilde u|^2+V(x)\widetilde u^2\right)\,dx
	+\sum_{i=1}^{k}\int_{\mathbb R^N}\left(\alpha|\nabla X^i|^2+|X^i|^2\right)\,dx .
\end{aligned}
\]

Using the limiting equations satisfied by $\widetilde u$ and $X^i$ in Step 2, we obtain
\begin{equation}
	\label{l1: eq-H-splitting}
	\lim_{n\to\infty}
	\int_{\mathbb R^N}H_n(|u_n|)\,dx
	=
	\int_{\mathbb R^N}H_0(|\widetilde u|)\,dx
	+
	\sum_{i=1}^{k}
	\int_{\mathbb R^N}H_0(|X^i|)\,dx .
\end{equation}

We next introduce the energy density required for the splitting of the energy. For $p>2$ and $s\geq0$, set
\[
K_p(s):=
\frac{\beta}{p}s^p
+m\gamma\int_0^s
\left(\frac{t^{p-2}-1}{p-2}\right)^{m-1}t^{p-1}\,dt,
\]
and define
\[
K_0(s):=
\frac{\beta}{2}s^2
+m\gamma\int_0^s(\ln t)^{m-1}t\,dt,
\qquad s>0,
\]
with $K_0(0):=0$. Since $m-1$ is even, $K_p(s)\geq0$ for all $s\geq0$.

We claim that there exists a constant $C>0$, independent of $p\in(2,p_0)$, such that
\begin{equation}
	\label{l1: eq-K-bound}
	0\leq K_p(s)\leq C\bigl(s^2+|H_p(s)|\bigr),
	\qquad s\geq0,
\end{equation}
where
\[
H_p(s):=\left[\beta\frac{s^{p-2}-1}{p-2}
+\gamma\left(\frac{s^{p-2}-1}{p-2}\right)^m\right]s^2.
\]
Indeed, by \eqref{l1: eq2-9},
\[
\int_0^s
\left(\frac{t^{p-2}-1}{p-2}\right)^{m-1}t^{p-1}\,dt
\leq
C s^p\left(1+\left|\frac{s^{p-2}-1}{p-2}\right|^{m-1}\right).
\]
For $0<s\leq1$, we have $s^p\leq s^2$ and
\[
\left|\frac{s^{p-2}-1}{p-2}\right|^{m-1}
\leq1+\left|\frac{s^{p-2}-1}{p-2}\right|^m.
\]
Since $m$ is odd, $H_p(s)\leq0$ on $(0,1]$, and hence
\[
s^2\left|\frac{s^{p-2}-1}{p-2}\right|^m
\leq C|H_p(s)|.
\]
Thus \eqref{l1: eq-K-bound} follows for $0<s\leq1$.
For $s>1$, using
\[
s^p=s^2\left(1+(p-2)\frac{s^{p-2}-1}{p-2}\right)
\]
and
\[
\left(\frac{s^{p-2}-1}{p-2}\right)^{m-1}
\leq1+\left(\frac{s^{p-2}-1}{p-2}\right)^m,
\]
we again obtain \eqref{l1: eq-K-bound}.

Let $\Omega_{n,R}$ be as in Step 3. By the strong convergence of $v_n$ obtained in Step 4 and the separation of the profiles,
\begin{equation}
	\label{l1: eq-L2-tail}
	\lim_{R\to\infty}\limsup_{n\to\infty}
	\int_{\Omega_{n,R}^c}|u_n|^2\,dx=0.
\end{equation}
Moreover, \eqref{l1: eq-H-splitting} and the local convergence established in Step 3 imply
\[
\lim_{R\to\infty}\lim_{n\to\infty}
\int_{\Omega_{n,R}^c}H_n(|u_n|)\,dx=0.
\]
On the other hand, the estimate in Step 3 gives
\[
\lim_{R\to\infty}\limsup_{n\to\infty}
\int_{\Omega_{n,R}^c}H_n^+(|u_n|)\,dx=0.
\]
Since $H_n=H_n^+-H_n^-$, it follows that
\begin{equation}
	\label{l1: eq-H-tail}
	\lim_{R\to\infty}\limsup_{n\to\infty}
	\int_{\Omega_{n,R}^c}|H_n(|u_n|)|\,dx=0.
\end{equation}
Therefore, by \eqref{l1: eq-K-bound}, \eqref{l1: eq-L2-tail}, and \eqref{l1: eq-H-tail},
\begin{equation}
	\label{l1: eq-K-tail}
	\lim_{R\to\infty}\limsup_{n\to\infty}
	\int_{\Omega_{n,R}^c}K_{p_n}(|u_n|)\,dx=0.
\end{equation}

For every fixed $R>0$, the local convergences obtained in Step 2 imply, after passing to a subsequence if necessary,
\[
K_{p_n}(|u_n|)\longrightarrow K_0(|\widetilde u|)
\quad\text{a.e. in }B_R(0),
\]
and, for each $i=1,\ldots,k$,
\[
K_{p_n}(|u_n(\cdot+y_n^i)|)\longrightarrow K_0(|X^i|)
\quad\text{a.e. in }B_R(0).
\]
By \eqref{l1: eq-K-bound}, the uniform integrability of $\{H_n(|u_n|)\}$ established in Step 3, and the local strong $L^2$ convergence, the corresponding sequences of $K_{p_n}$ are uniformly integrable on every fixed ball. Hence, by Vitali's theorem,
\[
\lim_{n\to\infty}
\int_{\Omega_{n,R}}K_{p_n}(|u_n|)\,dx
=
\int_{B_R(0)}K_0(|\widetilde u|)\,dx
+
\sum_{i=1}^{k}\int_{B_R(0)}K_0(|X^i|)\,dx.
\]
Combining this with \eqref{l1: eq-K-tail} and then letting $R\to\infty$, we conclude that
\begin{equation}
	\label{l1: eq-K-splitting}
	\lim_{n\to\infty}
	\int_{\mathbb R^N}K_{p_n}(|u_n|)\,dx
	=
	\int_{\mathbb R^N}K_0(|\widetilde u|)\,dx
	+
	\sum_{i=1}^{k}\int_{\mathbb R^N}K_0(|X^i|)\,dx.
\end{equation}

\noindent
\emph{Step 6. Splitting of the energy.}

For convenience, set
\[
L_p(s):=\frac{s^{p-2}-1}{p-2},\qquad s>0.
\]
By integration by parts, for every $s\geq0$,
\[
2\int_0^sL_p(t)t\,dt
=s^2L_p(s)-\frac{1}{p}s^p,
\]
and
\[
2\int_0^sL_p(t)^m t\,dt
=s^2L_p(s)^m
-m\int_0^sL_p(t)^{m-1}t^{p-1}\,dt.
\]
Consequently, from the definition of $I_p$, we obtain
\begin{equation}
	\label{l1: eq-energy-K}
	2I_p(u)-\langle I_p'(u),u\rangle
	=
	\int_{\mathbb R^N}K_p(|u|)\,dx.
\end{equation}
Since $u_n\in\mathcal N_{p_n}$,
\[
\langle I'_{p_n}(u_n),u_n\rangle=0,
\]
and hence
\begin{equation}
	\label{l1: eq-energy-Kpn}
	I_{p_n}(u_n)
	=
	\frac12\int_{\mathbb R^N}K_{p_n}(|u_n|)\,dx.
\end{equation}

Similarly, by the definitions of $I$, $I_\infty$, $\varphi$, and
$\varphi_\infty$, together with the corresponding logarithmic
integration-by-parts identities, we have
\[
2I(u)-\varphi(u)
=
\int_{\mathbb R^N}K_0(|u|)\,dx,
\]
and
\[
2I_\infty(u)-\varphi_\infty(u)
=
\int_{\mathbb R^N}K_0(|u|)\,dx.
\]
Since $\widetilde u$ is a weak solution of \eqref{l1: eq1-1} and each $X^i$ is a weak solution of \eqref{l1: eq2-3}, the self-testing identities established in Step 2 yield
\[
\varphi(\widetilde u)=0,
\qquad
\varphi_\infty(X^i)=0.
\]
Therefore,
\[
I(\widetilde u)
=
\frac12\int_{\mathbb R^N}K_0(|\widetilde u|)\,dx,
\]
and
\[
I_\infty(X^i)
=
\frac12\int_{\mathbb R^N}K_0(|X^i|)\,dx.
\]
Combining these identities with \eqref{l1: eq-K-splitting}, we obtain
\[
\begin{aligned}
	d
	=\lim_{n\to\infty}I_{p_n}(u_n)
	=\frac12\lim_{n\to\infty}
	\int_{\mathbb R^N}K_{p_n}(|u_n|)\,dx
	=\frac12\int_{\mathbb R^N}K_0(|\widetilde u|)\,dx
	+\frac12\sum_{i=1}^{k}
	\int_{\mathbb R^N}K_0(|X^i|)\,dx
	=I(\widetilde u)+\sum_{i=1}^{k}I_\infty(X^i).
\end{aligned}
\]
This completes the proof.
\qed

\begin{corollary}\label{l1: sec2-8}
\it Assume $(AC_2)$ holds. Let $w\in\mathcal D$ be a nonnegative ground state of the autonomous limit problem \eqref{l1: eq2-3}. Under the conditions and assumptions of \Cref{l1: sec2-7}, if $d\in(0,c_\infty)\cup(c_\infty,c_\infty+c)$, then $u_n\to\tilde u$ in $H^1(\R^N)$ up to a subsequence.
\end{corollary}

\noindent
{\bf Proof.} By \Cref{l1: sec2-7}, we have $d=I(\widetilde u)+\sum\limits_{i=1}^{k}I_\infty(X^i)$, where each $X^i$ is a nontrivial weak solution of the autonomous problem \eqref{l1: eq2-3}. By the self-testing identities established in \Cref{l1: sec2-7}, we have
\[
X^i\in\mathcal N_\infty,
\qquad
I_\infty(X^i)\ge c_\infty,
\qquad i=1,\ldots,k.
\]
Moreover, the energy identity established in Step 6 of \Cref{l1: sec2-7} gives
\[
I(\widetilde u)
=
\frac12\int_{\mathbb R^N}K_0(|\widetilde u|)\,dx
\ge0.
\]
If $\widetilde u\neq0$, then the self-testing identity also yields $\widetilde u\in\mathcal N$, and hence
\[
I(\widetilde u)\ge c.
\]

If $d\in(0,c_\infty)$ and $k\ge1$, then
\[
d
=
I(\widetilde u)+\sum_{i=1}^{k}I_\infty(X^i)
\ge c_\infty,
\]
which is impossible. Hence $k=0$, and \Cref{l1: sec2-7} yields
\[
u_n\to\widetilde u
\qquad\text{strongly in }H^1(\mathbb R^N).
\]

Next, suppose that
\[
d\in(c_\infty,c_\infty+c).
\]
If $k\ge2$, then, using \Cref{l1: sec2-6},
$$
d=I(\widetilde u)+\sum_{i=1}^{k}I_\infty(X^i)\ge I_\infty(X^1)+I_\infty(X^2)	\ge 2c_\infty\ge c_\infty+c,
$$
which contradicts $d<c_\infty+c$.

It remains to exclude the case $k=1$.
If $\widetilde u\neq0$, then
\[
d
=
I(\widetilde u)+I_\infty(X^1)
\ge c+c_\infty,
\]
again a contradiction.

If $\widetilde u=0$, then
\[
d=I_\infty(X^1).
\]
Since \Cref{l1: sec2-6} gives $c\le c_\infty$, we have
\[
c_\infty
<
d
<
c_\infty+c
\le2c_\infty.
\]
Thus
\[
d\in(c_\infty,2c_\infty),
\]
which contradicts \Cref{l1: sec2-1}, since $I_\infty$ has no critical value in $(c_\infty,2c_\infty)$.

Therefore $k=0$. By \Cref{l1: sec2-7},
\[
u_n\to\widetilde u
\qquad\text{strongly in }H^1(\mathbb R^N).
\]
The proof is completed. \qed
\vskip2mm
{\section{Existence of a nonnegative bound state}}
\setcounter{equation}{0}
\noindent
In this section, we prove the existence of a nontrivial nonnegative bound state for the logarithmic problem. The main difficulty is to recover the compactness of the Palais-Smale sequences due to the lack of compactness in the whole space. To overcome this difficulty, we first construct a suitable min-max level below the splitting threshold by using the barycenter technique. Then, combining the energy estimate obtained previously with the profile decomposition of Palais-Smale sequences, we obtain a critical point of the approximating functional and complete the proof of the main theorem.
\subsection{Construction of the min-max level}\noindent
In this subsection, we construct a suitable min--max level for the approximating functional. First, we establish the convergence of the ground state energy levels as $p\to2^+$. Then, by introducing the barycenter map and constructing an appropriate class of paths, we obtain a critical level below the threshold where the splitting of Palais--Smale sequences may occur.
\begin{proposition}
	\label{l1: sec3-1}
	Assume $(AC_1)$ holds. Then there exists $\lim\limits_{p\to2^+}c_p=c$. Moreover, if $c<c_\infty$, then $c$ is attained.
\end{proposition}
\noindent
\noindent
{\bf Proof.}Let $p_n\to2^+$. By the Ekeland variational principle and \Cref{l1: sec2-6}, there exists a subsequence of $u_n\in\mathcal N_{p_n}$, still denoted by $\{u_n\}$, such that $\|I'_{p_n}|_{\mathcal N_{p_n}}(u_n)\|_*\to0$ and $I_{p_n}(u_n)\to d \in (0,c]$. Applying \Cref{l1: sec2-7}, up to a subsequence, we obtain
\[
d
=
I(\widetilde u)
+
\sum_{i=1}^{k}I_\infty(X^i).
\]

If $\widetilde u\neq0$, then the self-testing identity gives
$\widetilde u\in\mathcal N$, and hence
\[
d\ge I(\widetilde u)\ge c.
\]
If $\widetilde u=0$, then $d>0$ implies $k\ge1$. Therefore,
\[
d
=
\sum_{i=1}^{k}I_\infty(X^i)
\ge c_\infty
\ge c,
\]
where the last inequality follows from \Cref{l1: sec2-6}.
Thus in either case $d\ge c$. Since $d\le c$, we conclude that
\[
d=c.
\]
Hence every subsequential limit of $c_{p_n}$ is equal to $c$, and
therefore
\[
\lim_{p\to2^+}c_p=c.
\]

Now assume that $c<c_\infty$. Let $p_n\to2^+$ and choose
$u_n\in\mathcal N_{p_n}$ by the Ekeland variational principle such that
\[
I_{p_n}(u_n)\to c,
\qquad
\big\|I'_{p_n}|_{\mathcal N_{p_n}}(u_n)\big\|_*
\to0.
\]
By \Cref{l1: sec2-7},
\[
c
=
I(\widetilde u)
+
\sum_{i=1}^{k}I_\infty(X^i).
\]
If $k\ge1$, then, since $I(\widetilde u)\ge0$,
\[
c
\ge I_\infty(X^1)
\ge c_\infty,
\]
which contradicts $c<c_\infty$. Hence $k=0$.

It follows from \Cref{l1: sec2-7} that
\[
u_n\to\widetilde u
\qquad\text{strongly in }H^1(\mathbb R^N),
\]
and
\[
I(\widetilde u)=c.
\]
Since $c>0$, we have $\widetilde u\neq0$. Moreover, the self-testing identity yields $\widetilde u\in\mathcal N$. Therefore $c$ is attained.\qed
\vskip2mm
\noindent

By \Cref{l1: sec3-1,l1: sec2-6}, we have $c\leq c_\infty$. If $c$ is attained by a nontrivial critical point of $I$, then the same positive-negative part argument as in the proof of \Cref{l1: sec1-1} yields a nontrivial nonnegative bound state, and the proof is completed.

Therefore, in the sequel, we only consider the case where $c$ is not attained. By \Cref{l1: sec3-1}, we then necessarily have $c=c_\infty$.

\vskip2mm
\noindent
\begin{definition}\label{l1: sec3-2}
\it Define the center of gravity $g: H^1(\R^N)\setminus \{0\}\to \R^N$ as follows:
$$
g(u)=\int_{\R^N}\frac{x}{|x|}u^2dx,\;\; u\in H^1(\R^N)\setminus\{0\}.
$$
\end{definition}
\begin{remark}\label{l1: sec3-3}
\it Denote $w(x-h)$ as the radial ground state solution of (\ref{l1: eq2-3}). Then there exists
$$
\lim_{|h|\to\infty}g(w(x-h))\cdot \frac{h}{|h|}=\int_{\R^N}w^2(x)\,dx=:C_7>0.
$$
\end{remark}
Define
$$
\mathcal{M}_p:=\big\{u\in \mathcal{N}_p\mid g(u)=0\big\},\;\;
b_p:=\inf_{u \in \mathcal{M}_p} I_p(u).
$$
Then, we have the following Lemmas.

\begin{lemma}\label{l1: sec3-4}
	For $p\to 2^+$, one has $\liminf\limits_{p\to2^+}b_p>c$.
\end{lemma}

\noindent
{\bf Proof.}
Since $\mathcal M_p\subset\mathcal N_p$, by the definitions of $b_p$ and $c_p$ we have $b_p\geq c_p$. Hence, in virtue of \Cref{l1: sec3-1}, $\liminf\limits_{p\to2^+}b_p\geq c$. 

Suppose by contradiction that
\[
\liminf_{p\to2^+}b_p=c.
\]
Then there exists a sequence $p_n\to2^+$ such that $b_{p_n}\to c$. It follows from \Cref{l1: sec3-1} that
\[
b_{p_n}-c_{p_n}\to0.
\]

Choose $z_n\in\mathcal M_{p_n}$ such that $I_{p_n}(z_n)\leq b_{p_n}+\frac1n$. Then $I_{p_n}(z_n)-c_{p_n}\to0$. Applying the Ekeland variational principle to $I_{p_n}|_{\mathcal N_{p_n}}$, we obtain $u_n\in\mathcal N_{p_n}$ such that
\[
I_{p_n}(u_n)\to c,
\qquad
\big\|I'_{p_n}|_{\mathcal N_{p_n}}(u_n)\big\|_*\to0,
\]
and
\[
\|u_n-z_n\|_{H^1(\mathbb R^N)}\to0.
\]
Since $z_n\in\mathcal M_{p_n}$, we have $g(z_n)=0$. Moreover, using
\[
|g(u)-g(v)|
\leq
\int_{\mathbb R^N}|u^2-v^2|\,dx
\leq
\|u-v\|_{L^2}
\big(\|u\|_{L^2}+\|v\|_{L^2}\big),
\]
together with the boundedness of $\{u_n\}$ and $\{z_n\}$, we obtain
\[
g(u_n)\to0.
\]

Applying \Cref{l1: sec2-7}, up to a subsequence, there exist a weak solution $\widetilde u$ of \eqref{l1: eq1-1}, nontrivial weak
solutions $X^1,\ldots,X^k$ of \eqref{l1: eq2-3}, and translations $y_n^i$ such that
\[
c
=
I(\widetilde u)
+
\sum_{i=1}^{k}I_\infty(X^i).
\]
Since $I(\widetilde u)\geq0$ and $I_\infty(X^i)\geq c_\infty=c$, we must have $k\leq1$.

If $k=0$, then $I(\widetilde u)=c$. Since $c>0$, $\widetilde u\neq0$, and hence $c$ is attained, which contradicts the standing non-attainment assumption. Therefore, $k=1$.

If $\widetilde u\neq0$, then $\widetilde u\in\mathcal N$, and hence $c=I(\widetilde u)+I_\infty(X^1)\geq c+c_\infty=2c$, which is impossible. Consequently, $\widetilde u=0$. Thus, $I_\infty(X^1)=c=c_\infty$, and \Cref{l1: sec2-7} gives $u_n-X^1(x-y_n^1) \to0\;\text{in }H^1(\mathbb R^N)$, where $|y_n^1|\to\infty$.

By the change of variables $x=z+y_n^1$ and the dominated convergence theorem,
\begin{eqnarray*}
	g\big(X^1(x-y_n^1)\big)\cdot \frac{y_n^1}{|y_n^1|}
	=
	\int_{\mathbb R^N}
	\frac{z+y_n^1}{|z+y_n^1|}
	\cdot \frac{y_n^1}{|y_n^1|}\,|X^1(z)|^2\,dz
	\longrightarrow
	\int_{\mathbb R^N}|X^1(z)|^2\,dz
	>0.
\end{eqnarray*}

On the other hand,
\[
\begin{aligned}
	\left|	g\big(X^1(x-y_n^1)\big)-g(u_n)\right|
	=\left|\int_{\R^N}\frac{x}{|x|}\Big[\big(X^1(x-y_n^1)\big)^2-u_n^2(x)\Big]dx\right|
	&\leq
	\int_{\mathbb R^N}
	\left|
	\big(X^1(x-y_n^1)\big)^2-u_n^2(x)
	\right|dx\\
	&\leq
	\left(
	\|X^1\|_{L^2}+\|u_n\|_{L^2}
	\right)
	\|X^1(\cdot-y_n^1)-u_n\|_{L^2}\\
	&\longrightarrow 0.
\end{aligned}
\]
Since $g(u_n)\to0$, it follows that
\[
g\big(X^1(x-y_n^1)\big)\to0,
\]
which contradicts the previous positive limit.

Therefore,
\[
\liminf_{p\to2^+}b_p>c.
\]
The proof is completed. \qed

\begin{lemma}\label{l1: sec3-5}
	\it For every $h\in\R^N$ and every $p>2$ sufficiently close to $2$,
	there exists a unique constant $t_{h,p}>0$ such that
	$t_{h,p}w(x-h)\in\mathcal{N}_p$. Moreover, there exist $p_0>2$ and
	a constant $T_0<2^{\frac1\beta}$ such that
	\[
	t_{h,p}<T_0
	\]
	for all $h\in\R^N$ and $2<p<p_0$. Furthermore,
	\begin{eqnarray}
		\label{l1: 3-1}
		&\lim\limits_{p\to 2^+,\,|h|\to\infty}t_{h,p}=1,\nonumber\\
		&\limsup\limits_{p\to 2^+,\,|h|\to\infty}
		I_p(t_{h,p}w(x-h))\leq c_\infty.
	\end{eqnarray}
\end{lemma}

\noindent
{\bf Proof.}
Let $h\in\R^N$. According to the definition of $w(x-h)$ in
\Cref{l1: sec3-3}, we have
\[
\alpha\int_{\R^N}|\nabla w(x-h)|^2dx+\int_{\R^N}w^2(x-h)dx
=
\beta\int_{\R^N}w^2(x-h)\ln|w(x-h)|dx+\gamma\int_{\R^N}w^2(x-h)(\ln|w(x-h)|)^m dx.
\]

For convenience, set
\[
L_p(s):=\frac{s^{p-2}-1}{p-2},
\qquad s\geq0.
\]
By the ray characterization of $\mathcal N_p$ established in Section 2.2, for every $h\in\R^N$ and every $p>2$ sufficiently close to $2$, there exists a unique $t_{h,p}>0$ such that
\[
t_{h,p}w(x-h)\in\mathcal N_p.
\]
Hence, by the Nehari identity,
$$
 \int_{\R^N}	\Big(\alpha|\nabla w(x-h)|^2+V(x)w^2(x-h)\Big)dx
=\int_{\R^N}\left[	\beta L_p(t_{h,p}w(x-h))+\gamma L_p(t_{h,p}w(x-h))^m\right]w^2(x-h)\,dx.
$$

Making the change of variables $x=y+h$, set
$$
A_h:=\alpha\int_{\R^N}|\nabla w|^2dx+\int_{\R^N}V(y+h)w^2(y)\,dy\;\;\text{and}\;\;
B_p(t):=\int_{\R^N}\left[\beta L_p(tw)+\gamma L_p(tw)^m\right]w^2dx.
$$
Then
\begin{equation}\label{l1: eq4-1}
	A_h=B_p(t_{h,p}).
\end{equation}

Since $m$ is odd, the function $s\longmapsto\beta s+\gamma s^m$ is strictly increasing on $\R$. Moreover, $t\mapsto L_p(tw(x))$ is nondecreasing for every $x$ and strictly increasing whenever $w(x)>0$. Since $w\not\equiv0$, it follows that $t\longmapsto B_p(t)$ is strictly increasing on $(0,\infty)$.

We first derive a uniform upper bound for $t_{h,p}$. Choose $T_0:=e^{\theta/\beta}$, where $\theta\in(\max\{V^*-1,0\},\ln2)$. It follows from $(AC_1)$ that
\[
1<T_0<2^{1/\beta}.
\]

Define
\[
B_0(t):=\int_{\R^N}\left[\beta\ln(tw)+\gamma(\ln(tw))^m\right]w^2dx,
\]
where the integrand is understood by continuous extension at points where $w=0$.
Since $w$ solves the autonomous problem \eqref{l1: eq2-3},
\[
\alpha\int_{\R^N}|\nabla w|^2dx+\int_{\R^N}w^2dx=B_0(1).
\]
Moreover, since $m$ is odd and $\ln T_0>0$, the function $s\mapsto s^m$ is strictly increasing, and hence
$$
B_0(T_0)-B_0(1)=\beta\ln T_0\int_{\R^N}w^2dx+\gamma\int_{\R^N}\Big[(\ln T_0+\ln w)^m-(\ln w)^m\Big]w^2dx
\geq\theta\int_{\R^N}w^2dx.
$$
On the other hand,
$$
A_h-B_0(1)=\int_{\R^N}(V(y+h)-1)w^2(y)\,dy\leq	(V^*-1)\int_{\R^N}w^2dy.
$$
Therefore,
\[
B_0(T_0)-A_h\geq\big[\theta-(V^*-1)\big]\int_{\R^N}w^2dx=: \delta>0
\]
uniformly with respect to $h\in\R^N$.

For the fixed number $T_0$, by the pointwise convergence
\[
\left[
\beta L_p(T_0w)+\gamma L_p(T_0w)^m
\right]w^2
\to
\left[
\beta\ln(T_0w)+\gamma(\ln(T_0w))^m
\right]w^2
\qquad\text{as }p\to2^+,
\]
where the expression on the right-hand side is understood to be zero at the zeros of $w$ by continuous extension. Together with the uniform growth estimates established in Section 3 and the dominated convergence theorem, we have
\[
B_p(T_0)\to B_0(T_0)
\qquad\text{as }p\to2^+.
\]
Hence, for $p>2$ sufficiently close to $2$,
\[
|B_p(T_0)-B_0(T_0)|<\frac{\delta}{2}.
\]
Consequently,
\[
B_p(T_0)-A_h
\geq
-\frac{\delta}{2}+\delta
=
\frac{\delta}{2}>0
\]
for every $h\in\R^N$. Thus,
\[
B_p(T_0)>A_h=B_p(t_{h,p}).
\]
By the strict monotonicity of $B_p$, there exists $p_0>2$ such that
\[
t_{h,p}<T_0<2^{1/\beta}
\]
for every $h\in\R^N$ and every $p\in(2,p_0)$.

We next prove the first limit in \eqref{l1: 3-1}. 
Let $p_n\to2^+,|h_n|\to\infty$ and set $t_n:=t_{h_n,p_n}$. By the preceding estimate, $\{t_n\}$ is uniformly bounded from above. We claim that it is also bounded away from zero. Indeed, set
\[
A_*:=\alpha\int_{\R^N}|\nabla w|^2dx+V_0\int_{\R^N}w^2dx>0.
\]
Since $m$ is odd and $w^2|\ln w|^m\in L^1(\R^N)$, we have $B_0(t)\to-\infty\;\text{as }t\to0^+$. Thus we may choose $t_*>0$ sufficiently small such that
\[
B_0(t_*)<A_*.
\]
Using the convergence $B_p(t_*)\to B_0(t_*)$, for $p>2$ sufficiently close to $2$, one has
\[
B_p(t_*)<A_*\leq A_h
\qquad\text{for every }h\in\R^N.
\]
Using \eqref{l1: eq4-1} and the strict monotonicity of $B_p$, we obtain
\[
t_{h,p}>t_*.
\]
Therefore, for all large $n$,
\[
t_*<t_n<T_0.
\]

Up to a subsequence, assume that
\[
t_n\to t_0\in[t_*,T_0].
\]
By $(AC_1)$ and the dominated convergence theorem,
\[
A_{h_n}
\to
\alpha\int_{\R^N}|\nabla w|^2dx
+\int_{\R^N}w^2dx
=
B_0(1).
\]
Moreover, since $t_n$ remains in the compact interval $[t_*,T_0]$, the same uniform growth estimates and dominated convergence give
\[
B_{p_n}(t_n)\to B_0(t_0).
\]
Passing to the limit in \eqref{l1: eq4-1}, we obtain
\[
B_0(t_0)=B_0(1).
\]
Since $B_0$ is strictly increasing, it follows that
\[
t_0=1.
\]
Thus every convergent subsequence has the same limit, and hence
\[
\lim_{p\to2^+,\,|h|\to\infty}t_{h,p}=1.
\]

For the last result of \eqref{l1: 3-1}, using $t_{h,p}\to1$, $V(x+h)\to1$, and the pointwise convergence of the corresponding energy densities, understood by continuous extension at the zeros of $w$, together with the uniform growth estimates in Section 3 and the logarithmic integrability of $w$, the dominated convergence theorem yields
\[
I_p(t_{h,p}w(x-h))
\to
I_\infty(w)
=
c_\infty
\qquad
\text{as }p\to2^+,\quad |h|\to\infty.
\]
In particular,
\[
\limsup_{p\to2^+,\,|h|\to\infty}
I_p(t_{h,p}w(x-h))
\leq c_\infty.
\]
Thus the proof is completed. \qed

By \Cref{l1: sec3-4,l1: sec3-5}, we have 
\[
2^{\frac1\beta}-T_0>0,
\qquad
\liminf_{p\to2^+}b_p-c>0,
\qquad
c>0.
\]
Hence, we may choose
\[
0<\mu<
\min\left\{
2^{\frac1\beta}-T_0,\,
\liminf_{p\to2^+}b_p-c,\,
c
\right\}.
\]
Then, by \Cref{l1: sec3-5} and the barycenter limit in \Cref{l1: sec3-3}, we can choose $R>0$ sufficiently large and $p_1\in(2,p_0)$ sufficiently close to $2$ such that, for every $2<p<p_1$,
\begin{eqnarray}
	\label{l1: eq3-2}
	&t_{h,p}<T_0<2^{\frac1\beta}-\mu
	\qquad\text{for all }h\in\mathbb R^N,\nonumber\\
	&\min\limits_{|h|=R}g(w(x-h))\cdot h>\frac12 C_7R,\nonumber\\
	&\min\limits_{|h|=R}t_{h,p}>\frac12,\nonumber\\
	&\max\limits_{|h|=R}I_p(t_{h,p}w(x-h))<	\liminf\limits_{p\to2^+}b_p-\mu.
\end{eqnarray}

\begin{proposition}\label{l1: sec3-6}
	\it For each $p\in(2,p_1)$, give a new min-max value $\bar c_p$
	which is defined as
	\[
	\bar c_p:=
	\inf_{\zeta\in\Theta_p}
	\max_{h\in\overline{B_R(0)}}I_p(\zeta(h)),
	\]
	where
	\[
	\Theta_p:=
	\left\{
	\zeta\in C(\overline{B_R(0)},\mathcal N_p)
	\mid
	\zeta(h)=t_{h,p}w(x-h)\ \text{if }|h|=R
	\right\}.
	\]
	Then, in virtue of $(AC_3)$, there exists
	\begin{eqnarray}
		\label{l1: eq3-3}
		c<\liminf_{p\to2^+}b_p
		\leq\liminf_{p\to2^+}\bar c_p
		\leq\limsup_{p\to2^+}\bar c_p
		<\Lambda_0c<2c.
	\end{eqnarray}
\end{proposition}

\noindent
{\bf Proof.}
According to \Cref{l1: sec3-4} and the definition of upper bound and lower bound, it is evident that the first and third inequalities in equation \eqref{l1: eq3-3} hold. So, we just need to prove the remaining inequalities.

Now, we prove the second inequality, that is, $b_p\leq\bar c_p$ for all $p\in(2,p_1)$. In virtue of the definition of these parameters, we know that to prove this conclusion, we only need to show that
\[
(g\circ\zeta)^{-1}(0)\neq\emptyset,\qquad \forall\,\zeta\in\Theta_p.
\]
Set a function which is defined as
\[
H(s,\cdot):=s(g\circ\zeta)+(1-s){\rm id},
\]
where $s\in[0,1]$ and ${\rm id}$ is the identity operator on $\R^N$. In virtue of \eqref{l1: eq3-2}, we know that for all
$h\in\partial B_R(0)$ and $s\in[0,1]$, there holds
\begin{eqnarray*}
	H(s,h)\cdot h
	&=&
	s g(t_{h,p}w(x-h))\cdot h+(1-s)|h|^2\\
	&=&
	s t_{h,p}^2g(w(x-h))\cdot h+(1-s)|h|^2\\
	&\geq&
	\frac18sC_7R+(1-s)R^2\\
	&=&
	R^2+s\left(\frac18C_7R-R^2\right).
\end{eqnarray*}
Denote $F_3(s)=R^2+s\left(\frac{1}{8}C_7R-R^2\right)$. Since $F_3$ is linear and $F_3(0)=R^2>0$, $F_3(1)=\frac18 C_7R>0$, we have $F_3(s)>0$ for all $s\in[0,1]$. So, we have $H(s,h)\cdot h>0$ for all $h\in\partial B_R(0)$ and $s\in[0,1]$. Then, it follows from the homotopy invariance of the Brouwer degree in \cite{Chang2005} that
$$
\deg(g\circ\zeta,B_R(0),0)=\deg(H(1,\cdot),B_R(0),0)=\deg(H(0,\cdot),B_R(0),0)=\deg(\text{id},B_R(0),0)\neq0.
$$
After that, according to the Kronecker existence theorem in \cite{Chang2005}, since $0\notin(g\circ\zeta)(\partial B_R(0))$ and the above degree is nonzero, we obtain $(g\circ\zeta)^{-1}(0)\neq\varnothing$. Hence, there exists $h_\zeta\in B_R(0)$ such that $g(\zeta(h_\zeta))=0$. Since $\zeta(h_\zeta)\in\mathcal N_p$, we have $\zeta(h_\zeta)\in\mathcal M_p$. Therefore,
\[
b_p
\leq I_p(\zeta(h_\zeta))
\leq
\max_{h\in\overline{B_R(0)}}I_p(\zeta(h)).
\]
Taking the infimum over $\zeta\in\Theta_p$, we obtain
\[
b_p\leq\bar c_p.
\]
So, the inequality $\liminf\limits_{p\to2^+}b_p\leq\liminf\limits_{p\to2^+}\bar c_p$ has been proved.

Next, we prove the last inequality. By translation invariance, the parameter defined in $(AC_3)$ can equivalently be written as
$$
C(\beta,m,w):=\sup_{\substack{h\in \overline{B_R(0)}\\0<a\leq 2^{\frac1\beta}}}
\frac{\displaystyle\int_{\R^N}\int_0^a
	\left(\ln|\tau|+\ln|w(x-h)|\right)^{m-1}w^2(x-h)\tau\,d\tau dx}
{\displaystyle a\int_{\R^N}\int_0^1
	\left(\ln|\tau|+\ln|w(x-h)|\right)^{m-1}w^2(x-h)\tau\,d\tau dx}.
$$
Since $m-1$ is even, $w\geq0$ and $w\not\equiv0$, the integrand, understood by the continuous-extension convention in $(AC_3)$ at the zeros of $w$, is nonnegative and not identically zero. Hence, the denominator is strictly positive and the ratio function is well defined.
Denote
\[
I(a,h):=\int_{\R^N}\int_0^a(\ln|\tau|+\ln|w(x-h)|)^{m-1}w^2(x-h)\tau\,d\tau dx.
\]
After the change of variables $\tau=as$, we have
\begin{eqnarray*}
	I(a,h)
	& = &a^2\int_{\mathbb R^N}\int_0^1(\ln |a|+\ln |s|+\ln |w(x-h)|)^{m-1}w^2(x-h)s\,ds\,dx\\
	&\le&a^2\int_{\mathbb R^N}\int_0^1C_{m-1}\left(\Big|\ln |a|\Big|^{m-1}+\Big|\ln |s|\Big|^{m-1}+\Big|\ln |w(x-h)|\Big|^{m-1}\right)w^2(x-h)s\,ds\,dx\\
	& = & a^2\Big|\ln|a|\Big|^{m-1}C_{m-1}\int_{\mathbb{R}^N}w^2(x-h)dx
	+a^2C_{m-1}\int_{\mathbb{R}^N}w^2(x-h)dx\int_0^1s\Big|\ln|s|\Big|^{m-1}ds\\
	&   & +\frac{a^2}{2}C_{m-1}\int_{\mathbb{R}^N}w^2(x-h)\Big|\ln |w(x-h)|\Big|^{m-1}dx.
\end{eqnarray*}
According to \Cref{l1: sec2-5}, \eqref{l1: eq2-1}, and \Cref{l1: sec3-3}, we obtain $I(a,h)=O(a^2|\ln a|^{m-1})$, uniformly for $h\in B_R(0)$ as $a\to0^+$. Hence, the function defined by the ratio $\frac{I(a,h)}{aI(1,h)}$ in the definition of $C(\beta,m,w)$ admits a continuous extension at $a=0$ by setting its value to be zero. Since it is continuous on the compact set $[0,2^{\frac1\beta}]\times\overline{B_R(0)}$, we have $C(\beta,m,w)<+\infty$.

Now choose the path
\[
\zeta(h)=t_{h,p}w(x-h),\qquad h\in\overline{B_R(0)}.
\]
Since the Nehari equation defining $t_{h,p}$ depends continuously on $h$ and is strictly increasing with respect to the scaling parameter, the uniqueness of $t_{h,p}$ implies that $h\mapsto t_{h,p}$ is continuous. Hence, $\zeta\in\Theta_p$. By the definition of $\bar c_p$, we have
\[
\bar c_p\leq\max_{h\in\overline{B_R(0)}}I_p(t_{h,p}w(x-h)).
\]
It follows from $t_{h,p}w(x-h)\in\mathcal N_p$ that
\begin{small}
	\begin{eqnarray*}
		I_p(t_{h,p}w(x-h))
		&=&
		I_p(t_{h,p}w(x-h))
		-\frac12
		\left\langle
		I'_p(t_{h,p}w(x-h)),
		t_{h,p}w(x-h)
		\right\rangle\\
		&=&
		\frac{\beta}{2p}
		\int_{\R^N}|t_{h,p}w(x-h)|^pdx
		+
		\frac{m\gamma}{2}
		\int_{\R^N}\int_0^{t_{h,p}}
		\left(
		\frac{|\tau w(x-h)|^{p-2}-1}{p-2}
		\right)^{m-1}
		|\tau w(x-h)|^{p-2}
		w^2(x-h)\tau\,d\tau dx.
	\end{eqnarray*}
\end{small}Based on \Cref{l1: sec3-5} and the uniform estimates in Section 3, the following convergence is uniform with respect to $h\in\overline{B_R(0)}$ as $p\to2^+$:
\begin{small}
	\begin{eqnarray*}
		&&\max_{h\in\overline{B_R(0)}}
		I_p(t_{h,p}w(x-h))\\
		&=&\max_{h\in\overline{B_R(0)}}
		\left\{
		\frac{\beta t_{h,p}^2}{4}
		\int_{\R^N}|w(x-h)|^2dx	+\frac{m\gamma}{2}
		\int_{\R^N}\int_0^{t_{h,p}}
		(\ln|\tau w(x-h)|)^{m-1}
		\tau|w(x-h)|^2\,d\tau dx
		\right\}
		+o(1),
	\end{eqnarray*}
\end{small}where $o(1)\to0$ as $p\to2^+$ uniformly with respect to $h\in\overline{B_R(0)}$. In virtue of \eqref{l1: eq3-2}, we have $t_{h,p}<2^{\frac1\beta}-\mu$. Since $m-1$ is even, the integrand in the second term is nonnegative. Therefore,
\begin{eqnarray*}
\max_{h\in\overline{B_R(0)}}I_p(t_{h,p}w(x-h))
	&\leq&
	\max_{h\in\overline{B_R(0)}}\left\{\frac{\beta(2^{\frac1\beta}-\mu)^2}{4}\int_{\R^N}|w(x-h)|^2dx\right.\\
	&&\left.+\frac{m\gamma}{2}\int_{\R^N}\int_0^{2^{\frac1\beta}-\mu}(\ln|\tau w(x-h)|)^{m-1}\tau|w(x-h)|^2\,d\tau dx\right\}
	+o(1).
\end{eqnarray*}
According to the definition of $C(\beta,m,w)$, we have
\begin{small}
	\begin{eqnarray*}
	\int_{\R^N}\int_0^{2^{\frac1\beta}-\mu}	(\ln|\tau w(x-h)|)^{m-1}\tau|w(x-h)|^2\,d\tau dx
	\leq C(\beta,m,w)\left(2^{\frac1\beta}-\mu\right)\int_{\R^N}\int_0^1(\ln|\tau w(x-h)|)^{m-1}\tau|w(x-h)|^2\,d\tau dx.
	\end{eqnarray*}
\end{small}
Consequently,
\begin{small}
	\begin{eqnarray*}
		\max_{h\in\overline{B_R(0)}}I_p(t_{h,p}w(x-h))
		&\leq&
		\max_{h\in\overline{B_R(0)}}\left\{\frac{\beta(2^{\frac1\beta}-\mu)^2}{4}\int_{\R^N}|w(x-h)|^2dx\right.\\
		&&\left.+C(\beta,m,w)\left(2^{\frac1\beta}-\mu\right)\frac{m\gamma}{2}\int_{\R^N}\int_0^1(\ln|\tau w(x-h)|)^{m-1}		\tau|w(x-h)|^2\,d\tau dx\right\}+o(1)\\
		&\leq&
		\max\left\{	\left(2^{\frac1\beta}-\mu\right)^2,\,C(\beta,m,w)\left(2^{\frac1\beta}-\mu\right)\right\}\\
		&&\times\max_{h\in\overline{B_R(0)}}\left\{\frac{\beta}{4}\int_{\R^N}|w(x-h)|^2dx
		+\frac{m\gamma}{2}	\int_{\R^N}\int_0^1(\ln|\tau w(x-h)|)^{m-1}\tau|w(x-h)|^2\,d\tau dx\right\}+o(1).
	\end{eqnarray*}
\end{small}
Since $w$ is a ground state solution of the autonomous problem \eqref{l1: eq2-3}, by the corresponding Nehari energy identity,
\[
\frac{\beta}{4}\int_{\R^N}|w(x-h)|^2dx+\frac{m\gamma}{2}\int_{\R^N}\int_0^1(\ln|\tau w(x-h)|)^{m-1}\tau|w(x-h)|^2\,d\tau dx
=I_\infty(w)=c_\infty.
\]
Hence,
\begin{eqnarray*}
	\limsup_{p\to2^+}\bar c_p
	\leq\max\left\{\left(2^{\frac1\beta}-\mu\right)^2,\,C(\beta,m,w)\left(2^{\frac1\beta}-\mu\right)\right\}c_\infty
	<\max\left\{2^{\frac2\beta},\,C(\beta,m,w)2^{\frac1\beta}\right\}c_\infty
	=\Lambda_0c_\infty.
\end{eqnarray*}
Since we are considering the non-attainment case, by \Cref{l1: sec3-1}, we have $c=c_\infty$. Therefore, in virtue of $(AC_3)$,
\[
\limsup_{p\to2^+}\bar c_p<\Lambda_0c<2c.
\]
Thus, the last inequality in \eqref{l1: eq3-3} has been proved. Then, the proof is completed.\qed

\subsection{Proof of the main theorem}
\noindent
In this subsection, we complete the proof of the main theorem. Based on the min-max level constructed above, we obtain a Palais-Smale sequence for the approximating functional. The compactness results established in the previous sections allow us to show the convergence of this sequence and obtain a nontrivial critical point. Finally, passing to the limit $p\to2^+$, we recover a nontrivial nonnegative bound state of the original logarithmic Schr\"odinger equation.

By \Cref{l1: sec3-6}, choose $p_n\to2^+$ such that
\begin{eqnarray*}
\label{l1: eq3-5}
\bar c_{p_n}\to\bar c\in(c,2c).
\end{eqnarray*}
Moreover, by \eqref{l1: eq3-2} and \Cref{l1: sec3-6}, for sufficiently
large $n$, we have
\begin{eqnarray}
\label{l1: eq3-6}
\max_{|h|=R}I_{p_n}(t_{h,p_n}w(x-h))<\bar c-\mu.
\end{eqnarray}
\begin{proposition}\label{l1: sec3-7}
	\it
Let $p_n\to2^+$ be such that $\bar c_{p_n}\to\bar c\in(c,2c)$. Then there exists $u_n\in\mathcal N_{p_n}$ such that
$$
I_{p_n}(u_n)\to\bar c,\quad \big\|I'_{p_n}|_{\mathcal N_{p_n}}(u_n)\big\|_*\to0.
$$
\end{proposition}

\noindent
{\bf Proof.}
Let $\varepsilon_n=|\bar c_{p_n}-\bar c|+n^{-1}$. By the definition of \(\bar c_{p_n}\), there exists $\zeta_n\in\Theta_{p_n}$ such that
\[
\max_{h\in \overline{B_R(0)}}
I_{p_n}(\zeta_n(h))
<
\bar c+\varepsilon_n .
\]

Suppose by contradiction that this claim is false. Then, by the quantitative deformation lemma \cite{Minimax}, there exists a deformation $\eta_n\in C(\mathcal N_{p_n},\mathcal N_{p_n})$ such that
\begin{eqnarray}
	\label{l1: eq3-7}
	\eta_{n}(u)=u,&\text{if}&I_{p_n}(u)<\bar{c}-2\varepsilon_n,
\end{eqnarray}
and
\begin{eqnarray}
	\label{l1: eq3-8}
	I_{p_n}(\eta_{n}(u))\leq\bar{c}-\varepsilon_n&\text{if}&I_{p_n}(u)\leq\bar{c}+\varepsilon_n.
\end{eqnarray}
Since $\varepsilon_n\to0$, for sufficiently large $n$, we have
$2\varepsilon_n<\mu$. Thus, it follows from \eqref{l1: eq3-6} that
\[
\max_{|h|=R}I_{p_n}(t_{h,p_n}w(x-h))
<
\bar c-\mu
<
\bar c-2\varepsilon_n.
\]
Therefore, by \eqref{l1: eq3-7},
\[
\eta_n\circ\zeta_n\in\Theta_{p_n}.
\] So, we can obtain that $\max\limits_{z\in \overline{B_R(0)}}I_{p_n}(\eta_n\circ\zeta_n(z))\geq\bar{c}_{p_n}>\bar{c}-\varepsilon_n$, which contradicts (\ref{l1: eq3-8}). Therefore, for every sufficiently large $n$, there exists $u_n\in\mathcal N_{p_n}$ such that
\[
I_{p_n}(u_n)\in
[\bar c-2\varepsilon_n,\bar c+2\varepsilon_n]
\]
and
\[
\big\|I'_{p_n}|_{\mathcal N_{p_n}}(u_n)\big\|_*<8\varepsilon_n.
\]
Since $\varepsilon_n\to0$, the desired conclusion follows. Then, the proof is completed.\qed

\vskip2mm
\noindent

\noindent
{\bf Proof of \Cref{l1: sec1-1}.}
By \Cref{l1: sec3-6,l1: sec3-7}, there exist $p_n\to2^+$ and $u_n\in\mathcal N_{p_n}$ such that
\[
I_{p_n}(u_n)\to\bar c\in(c,2c),
\qquad
\big\|I'_{p_n}|_{\mathcal N_{p_n}}(u_n)\big\|_*\to0.
\]
Since in the non-attainment case $c=c_\infty$, we have
\[
\bar c\in(c_\infty,c_\infty+c).
\]
Hence, by \Cref{l1: sec2-8},
\[
u_n\to u
\qquad\text{strongly in }H^1(\R^N),
\]
up to a subsequence. Moreover, by \Cref{l1: sec2-7}, $u$ is a nontrivial critical point of $I$ and
\[
I(u)=\bar c\in(c,2c).
\]
It remains to prove that $u$ does not change sign. Assume by contradiction that $u$ changes sign. Then $u^+\neq0$ and $u^-\neq0$. By the same cutoff and self-testing argument as that used in \Cref{l1: sec2-7}, we may test the equation with $u^+$ and $-u^-$. Hence, $\langle I'(u),u^+\rangle=0, \langle I'(u),-u^-\rangle=0$, which implies that $u^+,u^-\in\mathcal N$. Moreover, $I(u)=I(u^+)+I(u^-)\ge2c$, which contradicts $I(u)=\bar c<2c$. Hence $u$ does not change sign. Replacing $u$ by $-u$ if necessary, we may assume that $u\geq0$. Since $u\not\equiv0$, it is a nontrivial nonnegative bound state solution. In the special case $m=1$, the strong maximum principle \cite{Pucci1999,Pucci2007} further yields $u>0$ in $\R^N$. The proof is completed.\qed

\vskip3mm
\noindent{\bf Acknowledgements} This project is  supported by China Association for Science and Technology Youth Science and Technology Talent Cultivation Project 2025 Doctoral Student Special Program and Yunnan Fundamental Research Projects (grant No: 202301AT070465).

\vskip3mm
\noindent{\bf Data availability statement} Data sharing not applicable to this article as no datasets were generated or analyzed during the current study.
\vskip2mm
\section*{Declarations}
\noindent{\bf Conflict of interest} The authors state no conflict of interest.

\vskip2mm
\renewcommand\refname{References}

\end{document}